\documentclass[reqno, letterpaper, oneside]{article}
\usepackage{amsmath,amsfonts,amssymb,amsthm}

\usepackage[usenames,dvipsnames]{color} 

\usepackage{tikz}
\usetikzlibrary{decorations.pathreplacing}
\usepackage{bbm}
\usepackage{setspace}
\usepackage{geometry}
\usepackage[section]{placeins} 

\newtheorem{theorem}{Theorem}

\newtheorem{conjecture}{Conjecture}
\newtheorem{lemma}[theorem]{Lemma}
\newtheorem{claim}[theorem]{Claim}
\newtheorem{remark}[theorem]{Remark}
\newtheorem{problem}{Problem}
\theoremstyle{definition}
\newtheorem{example}[theorem]{Example}

\newcommand\eps{\varepsilon}
\renewcommand\le{\leqslant}
\renewcommand\ge{\geqslant}

\newcommand\E{{\mathbb E}}

\newcommand\Var{\operatorname{Var}} 
 
\renewcommand\Pr{{\mathbb P}}
\newcommand\prob[1]{\Pr\left(#1\right)}

\newcommand{\Gnp}[1][p]{G_{n,{{#1}}}}

\newcommand\cC{\mathcal{C}}
\newcommand\cD{\mathcal{D}}
\newcommand\cE{\mathcal{E}}

\newcommand\cH{{\mathcal H}}

\newcommand\cI{{\mathcal I}}

\newcommand\cK{\mathcal{K}}
\newcommand\cL{\mathcal{L}}

\newcommand\cS{{\mathcal S}}

\newcommand{\fS}{{\mathfrak S}}

\newcommand{\x}{{\mathbf x}} 
\newcommand{\y}{{\mathbf y}} 
 
\newcommand{\sX}{X^*}

\newcommand{\refT}[1]{Theorem~\ref{#1}}

\newcommand{\refL}[1]{Lemma~\ref{#1}}
\newcommand{\refR}[1]{Remark~\ref{#1}}
\newcommand{\refS}[1]{Section~\ref{#1}}
\newcommand{\refP}[1]{Problem~\ref{#1}}

\newcommand{\refE}[1]{Example~\ref{#1}}
\newcommand{\refF}[1]{Figure~\ref{#1}}
\newcommand{\refApp}[1]{Appendix~\ref{#1}}
\newcommand{\refConj}[1]{Conjecture~\ref{#1}}
\newcommand{\refCl}[1]{Claim~\ref{#1}}

\newcommand\set[1]{\ensuremath{\{#1\}}}
\newcommand\bigset[1]{\ensuremath{\bigl\{#1\bigr\}}}

\newcommand\biggset[1]{\ensuremath{\biggl\{#1\biggr\}}}

\newcommand\bigpar[1]{\bigl(#1\bigr)}
\newcommand\Bigpar[1]{\Bigl(#1\Bigr)}
\newcommand\biggpar[1]{\biggl(#1\biggr)}

\newcommand\bigsqpar[1]{\bigl[#1\bigr]}

\newcommand\biggsqpar[1]{\biggl[#1\biggr]}

\newcommand\bigcpar[1]{\bigl\{#1\bigr\}}
\newcommand\Bigcpar[1]{\Bigl\{#1\Bigr\}}
\newcommand\biggcpar[1]{\biggl\{#1\biggr\}}

\def\rompar(#1){\textup(#1\textup)}    

\def\xexp(#1){e^{#1}}
\newcommand\ceil[1]{\lceil#1\rceil}

\newcommand\Bigceil[1]{\Bigl\lceil#1\Bigr\rceil}

\newcommand{\al}[1]{\alpha^*_{#1}}

\newcounter{case}

\newcounter{thmenumerate}

\let\OLDthebibliography\thebibliography
\renewcommand\thebibliography[1]{
  \OLDthebibliography{#1}
  \setlength{\parskip}{0pt}
  \setlength{\itemsep}{0pt plus 0.3ex}
}

\begin{document}

\title{A Counterexample to the DeMarco-Kahn Upper Tail Conjecture} 
\author{Matas {\v{S}}ileikis\thanks{Department of Theoretical Computer Science, Institute of Computer Science of the Czech Academy of Sciences, 182~07~Prague, Czech Republic. 
E-mail: {\tt matas.sileikis@gmail.com}. With institutional support RVO:67985807. Research supported by the Czech Science Foundation, grant number GJ16-07822Y.}
 \ 
 and Lutz Warnke\thanks{School of Mathematics, Georgia Institute of Technology, Atlanta GA~30332, USA.
E-mail: {\tt warnke@math.gatech.edu}. Research partially supported by NSF Grant DMS-1703516 and a Sloan Research Fellowship.}}
\date{9 October 2018; revised February 8, 2019} 
\maketitle

\begin{abstract}
Given a fixed graph~$H$, what is the (exponentially small) probability that the number~$X_H$ of copies of~$H$ in the binomial random graph~$G_{n,p}$ is at least twice its mean? 
Studied intensively since the mid~1990s, this so-called infamous upper tail problem remains a challenging testbed for concentration inequalities. 
In~2011 DeMarco and Kahn formulated an intriguing conjecture about the exponential rate of decay of~$\Pr(X_H \ge (1+\eps)\E X_H)$ for fixed~$\eps>0$. 
We show that this upper tail conjecture is false, by exhibiting an infinite family of graphs violating the conjectured bound.
\end{abstract}

\section{Introduction}
Understanding the distribution of subgraph counts is one of the central topics in random graph theory. 
Ever since the seminal paper of Erd{\H{o}}s and R{\'e}nyi~\cite{ER1960} from 1960 it has served as a rich source of intriguing probabilistic challenges and conjectures --- 
repeatedly stimulating the development of new insights and tools in combinatorial probability theory (in particular concentration inequalities). 

In this note we focus on the tails of the number~$X_H=X_H(n,p)$ of copies of a fixed graph~$H$ 
in the binomial random graph~$G_{n,p}$, which have been intensively studied for decades. 
Indeed, in the~1980s the need for exponentially small tail probabilities emerged in applications, 
and the behaviour of the \emph{lower tail} $\Pr(X_H \le (1-\eps)\E X_H)$ was 
eventually resolved by the celebrated Janson's inequality~\cite{JLR87,J90,RW2012J,JWL}.  
In the early~1990s the need for also understanding the exponential decay of the \emph{upper tail}~$\Pr(X_H \ge (1+\eps) \E X_H)$ became evident, 
and since then the following `infamous' upper tail problem has proven to be much more challenging 
than its lower tail counterpart (see~\cite{UT} and~\cite{J90,RR94} as well as~\cite[Section~4.8]{Vu2002} and~\cite[Problem~6.1]{DL}). 
\begin{problem}[Upper tail problem for subgraph counts]\label{pr:UT} 
Given a fixed graph~$H$ with~$e_H \ge 1$ edges, 
determine for fixed~$\eps > 0$ and arbitrary~$p = p(n) \in (0,1)$ the order of magnitude of 
\begin{equation}\label{eq:UTP}
-\log \Pr(X_H \ge (1+\eps) \E X_H) .
\end{equation}
\end{problem}
\noindent 
In~2002 Janson, Oleszkiewicz and Ruci{\'n}ski~\cite{JOR} finally
determined the exponential rate of decay~\eqref{eq:UTP} up to a factor of~$O(\log (1/p))$.  
This breakthrough solved \refP{pr:UT} for constant edge-probabilities~$p \in (0,1)$, 
but closing the logarithmic gap for~$p=o(1)$ remained an elusive technical challenge.

Shortly after the upper tail problem was settled for triangles~$H=K_3$~\cite{C12,DKtriangles}, 
in~2011 DeMarco and Kahn solved the more general case of fixed-size 
cliques~$H=K_r$ with~$r \ge 3$~\cite{DKcliques}, 
and also formulated a plausible \emph{conjecture} on the general solution of \refP{pr:UT};  
see \refConj{conj:UT} below. 
This `upper tail conjecture' has been verified for large~$p=p(n)$ of form~$p \ge n^{-\delta_H}$ via large deviation machinery~\cite{CD16,LZ17,BGLZ17,Eldan16}, 
and for small~$p=p(n)$ of form~$p \le n^{-v/e} (\log n)^{C_H}$ for so-called strictly balanced 
graphs~$H$~\cite{Vu2000,Sileikis2012,LW16} (where~$e_F/v_F < e_H/v_H$ for any non-empty~$F \subsetneq H$); 
see also~\cite{RW15,Sileikis,Sileikis2012,Stars} for further supporting results.  
In fact, this conjecture was also described as \mbox{`likely to be true'} 
in the recent random graphs book by Frieze and Karo{\'n}ski~\cite[Section~5.4]{KFRG}. 

\pagebreak[1]

In this note we show that the $7$-year-old DeMarco--Kahn upper tail conjecture for subgraph counts \mbox{is \emph{false}}, by exhibiting an infinite family of graphs which violate 
the conjectured behavior of the upper tail~\eqref{eq:UTP}; see \refT{thm:UT:LB} below. 
On a conceptual level, our results 
shed new light on the upper tail behaviour for small edge-probabilities~$p=p(n)$, indicating 
that close to the threshold of appearance the reason for having `too many' copies of~$H$ can be more complicated than previously anticipated (see Sections~\ref{sec:discussion} and~\ref{sec:conclusion}). 
In retrospect this might perhaps not seem so surprising, 
taking into account that at the appearance threshold the limiting distribution of~$X_H$ can be quite complicated, as discovered in the 1980s~\cite{BB81,J87,BW89,AR90}.

\subsection{Main result}\label{sec:conjecture}
Turning to the details, we now formally state\footnote{In the spirit of earlier questions and examples in the area (see, e.g.,~\cite[Section~4]{Vu2001} and~\cite[Section~6]{DL}), DeMarco and Kahn formally stated~\cite[Conjecture~10.1]{DKcliques} for~$\eps=1$ only, tacitly assuming the necessary condition~$p \le (1+\eps)^{-1/e_H}$. 
The natural variant~\eqref{eq:conj:UT} for arbitrary fixed~$\eps>0$ is of course also attributed to them; cf.~\cite[Section~5.4]{KFRG} and~\cite[Section~4]{DKtriangles}. 
In~\eqref{eq:conj:UT:MG} we also use a simplified (but up to constant factors equivalent) definition of~$M_H$; 
cf.~\cite[(46)]{DKcliques} and~\cite[Theorem~1.5 and Remark~1.6]{JOR}.}
the upper tail conjecture from~\cite{DKcliques}, 
which proposes a compelling solution to \refP{pr:UT}.   
Let~$\mu_H := \E X_H$, $\sigma^2_H := \Var X_H$, and $m_H := \max_{F \subseteq H: v_F \ge 1} e_F/v_F$, as usual. 
As pointed out in~\cite{JOR,DKcliques,Sileikis,KFRG}, to avoid degenerate behaviour of the upper tail it is natural and convenient to assume 
(i)~that $p=p(n)$ is above 
the appearance threshold~$n^{-1/m_H}$ of~$H$,
and (ii)~that $(1+\eps) \E X_H$ is at most the number of copies of~$H$ in the complete graph~$K_n$, 
which is equivalent to~$(1+\eps)p^{e_H} \le 1$. 
\begin{conjecture}[DeMarco and Kahn, 2011]\label{conj:UT} %
Let~$H$ be a graph with~$e_H \ge 1$ edges.  
For fixed~$\eps>0$ and any~$p=p(n)$ with~$n^{-1/m_H} < p \le (1+\eps)^{-1/e_H}$ we have 
\begin{equation}\label{eq:conj:UT}
-\log \prob{X_H \ge (1+\eps) \E X_H} = \Theta\Bigpar{\min \bigset{ \Phi_H, \: M_H \log(1/p) }} ,
\end{equation}
where the implicit constants in~\eqref{eq:conj:UT} may depend on~$\eps$ and~$H$, 
with 
\begin{align}
\label{eq:PhiH}
\Phi_H= \Phi_H(n,p) &:= \min_{G \subseteq H: e_G \ge 1} \mu_G ,\\
\label{eq:conj:UT:MG}
M_H= M_H(n,p) &:= 
\begin{cases}  
	{\displaystyle \min_{G \subseteq H: e_G \ge 1} \mu_G^{1/\alpha^*_G}} & \quad \text{if } p < n^{-1/\Delta_H}, \\
	{\displaystyle n^2p^{\Delta_H}} & \quad  \text{if } p \ge n^{-1/\Delta_H},
\end{cases}
\end{align}
where $\Delta_G$ is the maximum degree of~$G$, and $\alpha^*_G$ is the fractional independence number\footnote{The fractional independence number is defined as $\alpha^*_G := \max \sum_{v \in V(G)} f(v)$, where maximum is taken over all functions $f: V(G) \to [0,1]$ satisfying $f(u)+f(v) \le 1$ for every edge~$\{u,v\} \in E(G)$; see, e.g.,~\cite[Appendix~A]{JOR}.\label{fn:alpha}
} of~$G$.
\end{conjecture}
%
%
\noindent 
One conceptual contribution of the above conjecture was to enhance the exponent~\eqref{eq:conj:UT} by the~$\Phi_H$ term, %
whose inclusion only matters for~$p \le n^{-1/m_H} (\log n)^{O(1)}$ unless~$\Delta_H=1$ holds 
(cf.~\cite[Remark~8.3]{JOR} and~\refS{sec:discussion}).

Our main result shows that \refConj{conj:UT} is false, 
by proving that there are infinitely many graphs~$H$ 
which violate the conjectured exponential rate of decay~\eqref{eq:conj:UT} 
close to the appearance threshold~$n^{-1/m_H}$. 
\begin{theorem}[Counterexamples to \refConj{conj:UT}]\label{thm:UT:LB} 
There is an infinite family~$\cH$ of graphs 
such that the following holds for any $H \in \cH$. There exists a constant $c_H>0$ such that for fixed~$\eps >0$ and any~$p=p(n) \in [0,1]$ with $n^{-1/m_H} \ll p \ll n^{-1/m_H} (\log n)^{c_H}$  
we have  
\begin{equation}\label{eq:thm:UT:LB}
-\log \prob{X_H \ge (1+\eps) \E X_H} = o\Bigpar{\min \bigset{ \Phi_H, \: M_H \log(1/p) }} .
\end{equation}
\end{theorem}
\begin{remark}\label{rem:UT:LB}
Given $H \in \cH$, for fixed~$\eps>0$ and any~$p=p(n) \in [0,1]$ with $n^{-1/m_H} \ll p \ll 1$ we also have 
\begin{equation}\label{eq:rem:UT:LB}
-\log \prob{X_H \ge (1+\eps) \E X_H} = o(\Phi_H)  .
\end{equation}
\end{remark}
\noindent 
The proof of \refT{thm:UT:LB} and \refR{rem:UT:LB} is given in \refS{sec:example}. 
As we shall see, it uses the family of graphs~$\cH$ illustrated in \refF{fig:cycle}, 
which are all connected and balanced (i.e., satisfy~$e_H/v_H =m_H$).  
\refR{rem:UT:LB} demonstrates that their upper tail probabilities are significantly larger than the lower tail probabilities 
for virtually all edge-probabilities~$p$ of interest, since~\cite{J90,JWL} gives under analogous assumptions that 
\[
-\log \prob{X_H \le (1-\eps) \E X_H} = \Theta(\Phi_H) .
\]
We find this complete separation of the decay of the two tails conceptually interesting 
(it was previously only known a bit above the appearance threshold, 
i.e., for~$n^{-1/m_H}\log n \ll p \ll 1$; see~\cite[Remark~8.3]{JOR}). 
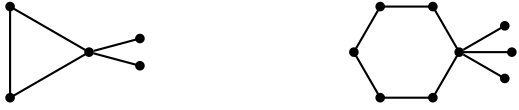
\begin{figure}[t]
\begin{center}
		\begin{tikzpicture}[thick,scale=0.7,decoration=brace]
    \tikzstyle{vertex}=[draw,circle,fill,fill opacity=1,minimum size=1pt, inner sep=1pt]
    \def\l{3}; 
    \def\r{2};
	\foreach \x in {1, 2, ..., \l}
	{
		\pgfmathsetmacro\angle{360/\l*(1 - \x)};
		\draw (\angle:1) node[vertex] (n\x) {} -- ( - 360/\l*\x:1) node () {};
	}
	\foreach \x in {1, 2, ..., \r} 
	{
		\pgfmathsetmacro\label{int(\x + \l)};
		\pgfmathsetmacro\angle{(\x - 1/2 - \r/2)*-30};
		\draw (n1) -- +(\angle:1) node[vertex] () {};
	}
    \begin{scope}[xshift=200]
    \def\l{6}; 
    \def\r{3};
	\foreach \x in {1, 2, ..., \l}
	{
		\pgfmathsetmacro\angle{360/\l*(1 - \x)};
		\draw (\angle:1) node[vertex] (n\x) {} -- ( - 360/\l*\x:1) node () {};
	}
	\foreach \x in {1, 2, ..., \r} 
	{
		\pgfmathsetmacro\label{int(\x + \l)};
		\pgfmathsetmacro\angle{(\x - 1/2 - \r/2)*-30};
		\draw (n1) -- +(\angle:1) node[vertex] () {};
	}
    \end{scope}
	  \end{tikzpicture}\vspace{-1.5em}
  \end{center}
  \caption{Examples 
	of the graph~$C_\ell^{+r}$ with~$(\ell,r)=(3,2)$ and~$(\ell,r)=(6,3)$, 
	obtained by attaching~$r$ pendant edges to some vertex of an~$\ell$-vertex cycle. 
	\refT{thm:UT:LB} shows that any graph in~$\cH := \set{C_\ell^{+r}: \: \ell \ge 3, r \ge 2}$ is a counterexample to the DeMarco--Kahn upper tail conjecture (see \refS{sec:example} for the full details).}
		\label{fig:cycle}
\end{figure}

\subsection{Discussion}\label{sec:discussion}
\refConj{conj:UT} can be interpreted as an educated guess to a variant of the following question: 
what is the most likely way to get at least $(1+\eps) \E X_H$ copies of~$H$ in~$G_{n,p}$?   
Indeed, as we shall see, it is based on two different mechanisms 
that each enforce~$X_H \ge (1+\eps) \E X_H$ in~$G_{n,p}$, 
giving two lower bounds with exponents~$M_H \log(1/p)$ and~$\Phi_H$. 
Hence~\eqref{eq:conj:UT} intuitively predicts that the dominating (more likely) mechanism determines the exponential decay of the upper tail, ignoring constant factors in the exponent.

The first \emph{clustered mechanism} is based on the idea that suitable `local' clustering of the edges can enforce many copies of~$H$ (e.g., a clique~$K_z$ contains~$\binom{z}{3} >2 \binom{n}{3}p^3$ triangles for suitable $z \asymp np$). 
In particular, if~$F \subseteq K_n$ contains at least $(1+\eps)\E X_H$ copies of~$H$, then by simply enforcing~$F \subseteq G_{n,p}$ we obtain 
\[
\Pr(X_{H} \ge (1+\eps) \E X_{H}) \ge \Pr(F \subseteq G_{n,p}) = p^{e_F} .
\] 
Janson, Oleszkiewicz and Ruci{\'n}ski~\cite{JOR} noted that one does not need to directly enforce copies of~$H$: it is enough if~$F \subseteq K_n$ contains unusually many copies of some subgraph~$J \subseteq H$ (say at least~$2(1+\eps)\E X_J$ many),  
since after planting~$F \subseteq G_{n,p}$ the rare upper tail event~$\set{X_H \ge (1+\eps)\E X_H}$ becomes `typical'. 
By minimizing the number~$e_F$ of edges over all such special graphs~$F \subseteq K_n$, 
this eventually gives 
a lower bound of form 
\begin{equation}\label{eq:lower:clustered}
\Pr(X_{H} \ge (1+\eps) \E X_{H}) \ge \max_{F \subseteq K_n} \Pr(X_{H} \ge (1+\eps) \E X_{H} \mid F \subseteq G_{n,p}) p^{e_F} \ge p^{\Theta(M_H)} ,
\end{equation}
see~\cite[Theorems~1.5 and~3.1]{JOR} for the full details. This explains the exponent~$M_H \log(1/p)$ in~\eqref{eq:conj:UT}.

The second \emph{disjoint mechanism} is based on many mutually exclusive `global' configurations of the edges, which each contain many disjoint copies of~$H$.  
Let~$\cD_{H,\eps}$ denote the event that~$G_{n,p}$ contains exactly~$k:=\ceil{(1+\eps)\E X_{H}}$ disjoint copies of~$H$ (either vertex-disjoint or edge-disjoint),  
and write~$N=N(n,H)$ for the number of~$H$-copies in~$K_n$. 
Summing over distinct $k$-sets of disjoint $H$-copies, 
for strictly-balanced graphs this eventually gives a binomial-like lower bound in some range, which turns out to be roughly of form 
\[
\Pr(X_{H} \ge (1+\eps) \E X_{H}) \ge \Pr(\cD_{H,\eps}) \approx \binom{N}{k} \cdot p^{k e_H} \cdot (1-p^{e_H})^{N-k} = \exp\Bigcpar{-\Theta(\mu_H)} , 
\]
see~\cite{DKcliques,Sileikis,Sileikis2012} for the full details. 
As in the clustered mechanism, it turns out that we can again optimize the resulting bound over all relevant subgraphs~$J \subseteq H$, 
eventually leading to a lower bound of form 
\begin{equation}\label{eq:lower:disjoint}
\Pr(X_{H} \ge (1+\eps) \E X_{H}) \ge \exp\Bigcpar{-\Theta(\Phi_H)}, 
\end{equation}
see~\cite[Theorem~4.4]{Sileikis} for the full details 
(note that, due to subgraphs consisting of a single edge, the optimization leading to~\eqref{eq:lower:disjoint} 
includes the mechanism which is based on enforcing~$G_{n,p}$ to have \mbox{`too many edges'}). 
This explains the exponent~$\Phi_H$ in~\eqref{eq:conj:UT}.  
%
%
Alternatively, by combining 
the intuition that~$X_H$ should have \emph{subgaussian tails} (in some range) 
with the standard variance estimate
\begin{equation}\label{eq:sigmaH}
\sigma_H^2 \asymp \mu_H^2/\Phi_H
\end{equation}
from~\cite[Lemma~3.5]{JLR}, 
for fixed~$\eps>0$ we again (heuristically) arrive at an exponent of order~$(\eps\mu_H)^2/\sigma^2_H \asymp \Phi_H$.

A key message of \refT{thm:UT:LB} and our counterexamples from Sections~\ref{sec:example}--\ref{sec:further} 
is that for some graphs~$H$ there is a yet another mechanism (which we tentatively call \emph{locally-disjoint mechanism}), 
whose lower bound can beat both aforementioned mechanisms close to the appearance threshold~$n^{-1/m_H}$.  
\refR{rem:UT:LB} also shows that for certain graphs the disjoint mechanism (with exponent~$\Phi_H$) never wins, 
complementing the known fact that the clustered mechanism (with exponent~$M_H\log(1/p)$) never wins for matchings~\cite{JOR,DKcliques}. 


\subsection{Organization}
The remainder of this note is organized as follows. 
In \refS{sec:example} we prove \refT{thm:UT:LB} and \refR{rem:UT:LB}, 
i.e., present a simple set of counterexamples and describe how they contradict the upper tail conjecture. 
In \refS{sec:further} we elaborate the basic idea: 
we describe a larger set of counterexamples, 
and also give a new lower bound for the upper tail. 
The final \refS{sec:conclusion} contains some concluding remarks and conjectures.

\section{Simple counterexamples: Proof of \refT{thm:UT:LB} and \refR{rem:UT:LB}}\label{sec:example} %
In this section we prove \refT{thm:UT:LB} and \refR{rem:UT:LB} 
by considering the  graphs~$H=C_\ell^{+r}$ illustrated in \refF{fig:cycle}, 
which are constructed from an~$\ell$-vertex cycle $C_\ell$ by connecting~$r$ additional vertices 
to the same vertex of the cycle (so~$v_H = e_H = \ell + r$). 
These graphs have a history of exemplifying  non-trivial behaviour of subgraph counts: 
(i)~in~1987 Janson used~$C_\ell^{+2}$ to demonstrate that at the threshold~$X_H$ can converge to complicated distributions~\cite[Section~10]{J87}, 
and (ii)~in~2000 Janson and Ruci\'nski used~$C_3^{+3}$  to demonstrate that near the threshold~$X_H$ need not always have subgaussian tails~\cite[Example~6.14]{DLP}. 
As we shall see, the following auxiliary result demonstrates yet another non-trivial behaviour of the graphs~$H=C_\ell^{+r}$, 
since the lower bound~\eqref{eq:lem:CE:fixed} will contradict \refConj{conj:UT} (and establish \refT{thm:UT:LB}). 
Note that~$m_H=1$ and~$\mu_H \asymp (np)^{\ell + r}$. 
\begin{lemma}\label{lem:CE:fixed}
Given integers~$\ell \ge 3$ and~$r \ge 1$, let~$H := C_\ell^{+r}$ be the graph defined above. 
For fixed~$\eps > 0$ and any~$p=p(n) \in [0,1]$ with $1 \ll np \ll n^{1/(1+\ell/r)}$ we have
\begin{equation}\label{eq:lem:CE:fixed}
\prob{X_H \ge (1+\eps) \E X_H} \ge \exp\biggcpar{- O\Bigpar{\mu_H^{1/r}\log(np)}} ,
\end{equation}
where the implicit constant in~\eqref{eq:lem:CE:fixed} may depend on~$\eps$ and~$H$. 
\end{lemma}
\noindent 
One basic strategy for proving lower bounds is to enforce~$F \subseteq G_{n,p}$ for some graph~$F$ which itself contains at least $(1+\eps)\mu_H$ copies of~$H$. 
For example, $F:= C_\ell^{+z}$ contains $\binom{z}{r} \ge (1+\eps)\mu_H$ copies of $H=C_\ell^{+r}$ for suitable~$z \asymp (\mu_H)^{1/r}$. 
Following~\cite{Vu2001,JOR}, by enforcing~$F$ \emph{on the first~$v_F$ vertices} of $G_{n,p}$ we would obtain 
\[
\Pr(X_H \ge (1+\eps)\mu_H) \ge \Pr(F \subseteq G_{n,p}) \ge p^{e_F} \ge \exp\biggcpar{- O\Bigpar{\mu_H^{1/r}\log(1/p)}} .
\]
Here we shall improve the~$\log(1/p) \asymp \log n$ in the exponent to~$\log (np)$ by enforcing~$F$ \emph{somewhere} in~$G_{n,p}$. 
To this end, much in the spirit of a sequential embedding idea from~\cite{BB81}, 
we will below use a two-round exposure of the edges of~$G_{n,p}$ to first find a `random' copy of~$C_\ell$, which we then extend to a copy of~$F= C_\ell^{+z}$. 
We believe that, for~$r \ge 2$ and~$\eps=\Theta(1)$, the resulting rate of decay~\eqref{eq:lem:CE:fixed} is best possible when~$np \to \infty$ slowly.  
\begin{proof}[Proof of \refL{lem:CE:fixed}]
Set~$p_2 := p/2$, and pick~$p_1 \in [p/2,p]$ such that~$(1 - p_1)(1 - p_2) = 1 - p$. 
We expose the edges in two rounds: for $i \in [2]$ we insert each of the $\binom{n}{2}$ possible edges into~$\cE_i$ independently with probability~$p_i$; 
their union~$\cE_1 \cup \cE_2$ then gives~$\Gnp$. 
To establish the lower bound~\eqref{eq:lem:CE:fixed}, the strategy is to 
(i)~first use the $\cE_1$--edges to find one copy~$G'$ of~$G:=C_\ell$,  
and (ii)~then use the $\cE_2$--edges to extend~$G'$ to at least~$\binom{z}{r} \ge (1+\eps)\mu_H$ copies of~$H$, 
by enforcing that (in~$\cE_2$) one vertex of $G'$ has~$z$ neighbours outside of~$V(G')$, where 
\begin{equation*}
z := \Bigceil{r \bigpar{(1+\eps) \mu_H}^{1/r}} \asymp (np)^{1+\ell/r} = o(n) .
\end{equation*}

Turning to the details, for step~(i) let~$\sX_G$ be the number of copies of~$G=C_{\ell}$ in~$\cE_1$. 
Since~$m_G=1$ and~$p_1 \ge p/2 \gg n^{-1}$, it is well-known (see, e.g.,~\cite[Theorem~3.4]{JLR}) that 
\begin{equation}\label{eq:CE:C}
\Pr(\sX_G \ge 1) = 1-o(1) .
\end{equation}

For step~(ii), we henceforth condition on the edge-set~$\cE_1$, and assume that~$\sX_G \ge 1$; 
we also fix a copy~$G'$ of $C_{\ell}$ in~$\cE_1$, and one vertex~$v \in V(G')$. 
Defining~$Z$ as the number of vertices in $[n] \setminus V(G')$ that are neighbours of~$v$ in~$\cE_2$,  
note that $Z = z$ implies $X_H \ge \binom{Z}{r} \ge (Z/r)^r \ge (1+\eps) \mu_H$. 
Hence 
\begin{equation}\label{eq:CE:prdisjoint}
\Pr(X_H \ge (1+\eps) \mu_H \mid \cE_1) \ge \Pr(Z = z \mid \cE_1) = \binom{n-\ell}{z}(p_2)^z (1-p_2)^{n-\ell-z} \ge \Bigpar{\frac{np}{4z}}^{z} e^{-np} \ge (np)^{-O(z)} . 
\end{equation}
It follows that $\Pr(X_H \ge (1+\eps) \mu_H \mid \sX_G \ge 1) \ge (np)^{-O(z)}$, 
which together with~\eqref{eq:CE:C} implies inequality~\eqref{eq:lem:CE:fixed}. 
\end{proof}
\noindent 
It is easy to check that the exponent of~\eqref{eq:lem:CE:fixed} is of order~$[(1+\eps)\mu_H]^{1/r}\log\bigl[(1+\eps)^{1/\ell}np\bigr]$. 
In \refS{sec:further} we will give a variant of the above argument which not only gives a better dependence on~$\eps$ (when~$\eps \to 0$), but also applies to a significantly larger family of graphs~$H$. 
We are now ready to prove \refT{thm:UT:LB} and \refR{rem:UT:LB}. 
%
%
\begin{proof}[Proof of \refT{thm:UT:LB}] 
Define~$\cH := \set{C_\ell^{+r}: \: \ell \ge 3, r \ge 2}$. 
We henceforth fix~$H=C_\ell^{+r} \in \cH$. 
Since every subgraph of $H$ with fewer than $\ell$ vertices is acyclic, 
for $1 \ll np \ll n^{1/(\ell-1)}$ we have 
\begin{equation}\label{eq:thm:UT:LB:PhiH}
	\Phi_H = \min_{ G \subseteq H : e_G \ge 1 } \mu_G \asymp \min\biggcpar{\min_{2 \le k \le \ell -1} \bigcpar{n^kp^{k-1}}, \: \min_{\ell \le k \le \ell + r}\bigcpar{ n^kp^k}} \asymp (np)^\ell \gg 1.
\end{equation}
Turning to the parameter~$M_H$ defined in~\eqref{eq:conj:UT:MG}, note that~$n^2p^{\Delta_H} \ge n$ for~$p \ge n^{-1/\Delta_H}$.
Since~$\alpha^*_G \le v_G \le \ell + r$ holds by definition (cf.~Footnote~\ref{fn:alpha} on page~\pageref{fn:alpha} or~\cite[Appendix~A]{JOR}), 
using~\eqref{eq:thm:UT:LB:PhiH} it follows that 
\begin{equation}\label{eq:thm:UT:LB:MH}
M_H \log(1/p) 
=\Omega\bigpar{\min\bigcpar{\Phi_H^{1/(\ell+r)}, \: n}} \cdot \log(1/p)  \gg \log n. 
\end{equation}
Using~$\ell \ge 3$ and~$r \ge 2$, it now is routine to check that there is a constant $c_H>0$ such that
	\begin{equation}\label{eq:thm:UT:LB:comp}
\mu_H^{1/r}\log(np) \asymp (np)^{1 + \ell/r}\log(np) \ll \min\bigset{(np)^{\ell}, \: \log n} 
	\end{equation}
for $1 \ll np \ll (\log n)^{c_H}$, 
which in view of~\eqref{eq:thm:UT:LB:PhiH}--\eqref{eq:thm:UT:LB:MH}, \refL{lem:CE:fixed} and~$m_H=1$ implies inequality~\eqref{eq:thm:UT:LB}. 
\end{proof}
\begin{proof}[Proof of \refR{rem:UT:LB}]
Note that the above proof shows~$\mu_H^{1/r}\log(np) \ll \Phi_H$ for~$1\ll np \ll n^{1/(\ell-1)}$, 
so~\refL{lem:CE:fixed} implies~\eqref{eq:rem:UT:LB} for~$n^{-1/m_H} \ll p \ll n^{-1/m_H+1/\ell}$, say (recall that~$m_H=1$ and~$r \ge 2$).  
In the remaining range of~$p$ then~\cite[Remark~8.3]{JOR} already states that inequality~\eqref{eq:rem:UT:LB} holds (even for~$n^{-1/m_H}\log n \ll p \ll 1$). 
\end{proof}


\section{Extensions and generalizations}\label{sec:further}
In this section we generalize the lower bound construction from \refS{sec:example}. 
First, in~\refS{sec:proofsketch} we show that many graphs~$H$ are not only counterexamples to~\refConj{conj:UT}, but also fail to have subgaussian upper tails in some range.  
Next, in~\refS{sec:covering:lowerbound} we state a new lower bound for the upper tail, which complements the two clustered/disjoint mechanism based lower bounds from \refS{sec:discussion} used in \refConj{conj:UT}.   
We believe that our new lower bounds will not only serve as a testbed for future refinements of the upper tail conjecture (cf.~\refS{sec:conclusion}), 
but also stimulate the development of new upper bounds (here the importance of having non-trivial lower bounds was already highlighted by Vu~\cite[Section~4.8]{Vu2002} more than 15~years ago).

To state our results, we now introduce some terminology on the structure of the graph~$H$. 
We say that a subgraph~$G \subseteq H$ is \emph{primal (for $H$)} if~$e_G/v_G = m_H$. 
Clearly all primal subgraphs are induced, and thus we can treat them as a family~$\cL_H$ of subsets of~$V(H)$; see \refCl{cl:thm:CE} below for further properties. 
We say that~$G_2$~\emph{covers} a~primal~$G_1$ if~$G_1 \subsetneq G_2$ and there is no further primal~$F$ with~$G_1 \subsetneq F \subsetneq G_2$. 

\subsection{Further counterexamples and a general lower bound construction}\label{sec:proofsketch} 
The first inequality~\eqref{eq:thm:CE:DK} of the following result generalizes \refT{thm:UT:LB}, 
by showing that many graphs~$H$ violate \refConj{conj:UT}. 
The second inequality~\eqref{eq:thm:CE:SG} conceptually generalizes \refR{rem:UT:LB}, 
by showing that the upper tail of these graphs is also not of a subgaussian type,  
no matter how close~$p=p(n)$ is to the appearance threshold~$n^{-1/m_H}$  
(even if we allow $\eps \to 0$ reasonably slowly; for~$H = C_3^{+3}$ this was already shown in~\cite[Example~6.14]{DLP}). 
The assumption~$\lambda \sigma_H \le t = O(\mu_H)$ below means that we are considering large deviations, i.e., deviations that are of higher order than the standard deviation~$\sigma_H=\sqrt{\Var X_H}$. 
\begin{figure}[t]
\begin{center}
\begin{tikzpicture}[thick,scale=1.3,decoration=brace]
			\pgfdeclarelayer{bg}    
			\pgfsetlayers{bg,main} 
    \tikzstyle{vertex}=[draw,circle,fill,fill opacity=1,minimum size=1pt, inner sep=1pt]
	\foreach \x in {1, 2, ..., 5}
	{
		\draw (108 - 72*\x:1) node[vertex,label={108 - \x*72:$x_\x$}] (n\x) {} -- (36 - 72*\x:1) node () {};
	}
	\foreach \x/\y in {1/3,2/4,3/5,4/1,5/2}
	{
		\draw (n\x) -- (n\y);
	}
	\foreach \x/\y in {0.7/6,0/7,-0.7/8} 
	{
		\draw (n1) -- (2,\x) node[vertex, label={above right:$x_\y$}] (n\y) {} -- (n2);
	}
	\draw (n6) -- (3,0.7) node[vertex, label={above right:$x_9$}] () {} -- (3,-1) node[vertex, label={above right:$x_{10}$}] () {} -- (n3);
	\foreach \a/\b/\y/\gr in {1/-1/-1.5/G, 2/-1/-1.9/K, 3/-1/-2.3/H} 
	{
	\draw [decorate] (\a,\y) to node[auto] {$\gr$} (\b,\y);
	}
\end{tikzpicture}
\vspace{-1.5em}
\end{center}
  \caption{Example of a graph with~$r = 3$, $G=H[\set{x_1, \ldots, x_5}])$, $J_i = H[V(G) \cup \{x_{5+i}\}])$, and~$K=J_1 \cup J_2 \cup J_3$, where~$m(H) = e_G/v_G = e_{J_i}/v_{J_i} = e_K/v_K=2$ and~$e_H/v_H = 19/10 < 2$.
\refT{thm:CE:cor} shows that this graph is another counterexample to the DeMarco--Kahn upper tail conjecture (and also fails to have subgaussian upper tails in some range).}
		\label{fig1}
\end{figure}
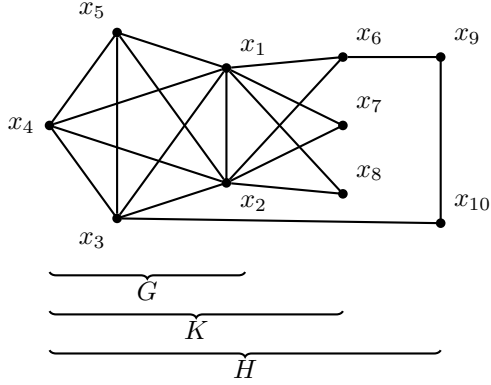
\begin{theorem}\label{thm:CE:cor}
Suppose that there is~$G \in \cL_H$ and distinct~$J_1, \dots, J_r \in \cL_H$ covering~$G$, such that~$K := J_1 \cup \cdots \cup J_r$ satisfies~$v_K/r < \min_{F \in \cL_H} v_F$. 
Then there are constants $c_H,\beta_H>0$ such that the following holds. 
For fixed $\eps>0$ and any~$p=p(n) \in [0,1]$ with~$1 \ll np^{m_H} \ll (\log n)^{c_H}$ we have 
\begin{equation}\label{eq:thm:CE:DK}
-\log \prob{X_H \ge (1+\eps) \E X_H} = o\Bigpar{\min \bigset{ \Phi_H, \: M_H \log(1/p) }} .
\end{equation}
Furthermore, there is~$\lambda = \lambda(n,p,H) \gg 1$ with~$\lambda \sigma_H \ll \mu_H$ such that, whenever~$1 \ll np^{m_H} \ll n^{\beta_H}$ and~$\lambda \sigma_H \le t = O(\mu_H)$ holds, we have 
\begin{equation}\label{eq:thm:CE:SG}
-\log \prob{X_H \ge \E X_H+t} = o\bigpar{t^2/\sigma^2_H} .
\end{equation}
\end{theorem}
\noindent 
Before giving the proof, we first use \refT{thm:CE:cor} to argue that counterexamples to \refConj{conj:UT} are abundant, by describing an abstract way of generating them. 
Suppose that we have a balanced graph~$J$ and a primal subgraph~$G$ with the property that~$J$ covers~$G$ (using as illustration \refF{fig1}, consider $G \cong K_5$ and construct~$J$ by connecting two vertices of~$G$ to a common outside neighbour). 
Then we construct~$K$ by `gluing'~$r$ distinct copies $J_1, \ldots, J_r$ of~$J$ in a consistent way\footnote{To make the gluing precise, writing $V(G)=\set{u_1, \ldots ,u_{v_G}}$ and $V(J) \setminus V(G) = \set{w_1, \ldots, w_{v_J-v_G}}$,  
the vertex-set of~$K$ consists of $V(G)$ and $r$ new vertices~$\{w_{j,1}, \dots, w_{j,r}\}$ for each $w_j \in V(J) \setminus V(G)$. 
The edge-set of~$K$ consists of $E(G)$ and~$\bigset{\set{u_i, w_{j,k}}: k \in [r]}$ for every $\set{u_i,w_j} \in E(J) \setminus E(G)$
as well~$\bigset{\set{w_{i,k}, w_{j,k}}: k \in [r]}$ for every~$\set{w_i,w_j} \in E(J) \setminus E(G)$.} 
onto~$G = \bigcap_{i \in [r]} J_r$ (see \refF{fig1} for an example with~$r=3$).
Let~$P$ be a primal of~$J$ with the minimum number of vertices (in \refF{fig1} we have~$P=G$). 
The resulting graph~$K$ is easily seen\footnote{For a formal proof of claims~(i)--(ii) 
note that, for any $Q \subseteq K$ with~$v_Q \ge 1$, using~$m_{J_i}=m_J=m_G = e_G/v_G$ we have 
\[
e_Q = e_{Q \cap G} + \textstyle\sum_{i \in [r]}\bigpar{e_{G \cup (Q \cap J_i)}-e_G} \le m_G \bigsqpar{v_{Q \cap G} + \textstyle\sum_{i \in [r]}\bigpar{v_{G \cup (Q \cap J_i)}-v_G}} = m_J v_Q ,
\] 
which holds with equality for~$Q=K$ and thus establishes~(i). 
For any primal~$Q \subseteq K$ the above inequality must also hold with equality, 
and in view of~$(Q \cap J_i) \cap G = Q \cap G$ it follows that~$e_{Q \cap J_i}=e_{G \cup (Q \cap J_i)}+e_{Q \cap G}-e_G = m_J v_{Q \cap J_i}$ for all~$i \in [r]$. 
Hence any~$Q \cap J_i \neq \emptyset$ (at least one such~subgraph must exist) is a primal of~$J_i \cong J$, so~$v_Q \ge v_{Q \cap J_i} \ge v_P$ establishes~(ii).} 
to (i)~be balanced with density~$m_J$, 
and (ii)~have no primal with fewer vertices than~$P$. If~$v_J - v_G < v_P$ holds, then $v_K/r = v_J - v_G + v_G/r < v_P$ for sufficiently large~$r$, 
in which case \refT{thm:CE:cor} implies that~$H:=K$ is a counterexample to \refConj{conj:UT}
(in~fact, this is true for any graph~$H \supseteq K$ for which~$P$ remains a vertex-minimal primal, as in \refF{fig1}).

We shall prove \refT{thm:CE:cor} as a corollary of the following more 
general result, which qualitatively extends \refL{lem:CE:fixed} to any graph~$H$ that is not strictly balanced 
(and also allows for $\eps \to 0$). 
Here we are again considering large deviations, since by~\eqref{eq:sigmaH} the  
assumption~$\eps^2\Phi_H \gg 1$ is equivalent to~$\eps \E X_H \gg \sqrt{\Var X_H}$. 
\begin{lemma}\label{lem:CE}
For any graph~$H$ with~$e_H \ge 1$ there is a constant $\beta_H>0$ such that the 
following holds for all~$\eps = \eps(n) > 0$ and~$p=p(n) \in [0,1]$ 
with $\eps^2 \Phi_H \gg 1$, $\eps=O(1)$, and~$1 \ll n p^{m_H}  \le n^{\beta_H}$. 
If~$G \in \cL_H$ and distinct~$J_1, \dots, J_r \in \cL_H$ cover~$G$, 
then we have, writing~$K := J_1 \cup \cdots \cup J_r$, 
\begin{equation}\label{eq:thm:CE}
\prob{X_H \ge (1 + \eps)\E X_H} \ge  \exp \biggcpar{ - O\Bigpar{ (\eps\mu_K)^{1/r} \log\bigpar{n p^{m_H}} }},
\end{equation}
where the implicit constant in~\eqref{eq:thm:CE} may depend on~$H$. 
\end{lemma}
\begin{remark}\label{rem:thm:CE}
The proof shows 
that for~$\eps=\Theta(1)$ the condition $\eps^2\Phi_H \gg 1$ is redundant (as in~\refL{lem:CE:fixed},  where~$m_H=1$), 
and that 
$\Phi_H \asymp (np^{m_H})^{\min_{F \in \cL_H}  v_F}$ 
holds for~$1 \ll n p^{m_H}  \le n^{\beta_H}$ (cf.~inequalities~\eqref{eq:CE:density:F}--\eqref{eq:CE:density}). 
\end{remark}
\noindent 
Refining the proof strategy of \refL{lem:CE:fixed}, inspired by~\cite{Vu2001,JOR,JWL,APUT} 
the idea is to first enforce~$y=\Theta(\eps \mu_K)$ copies of~$K$ via some some special~$F \subseteq G_{n,p}$,  
which we again find via two exposure rounds. 
Then we simultaneously (a)~extend these~$y$ copies of~$K$ to~$2\eps\mu_H$ copies of~$H$, 
and (b)~also find additional $(1-\eps) \mu_H$ `random' copies of~$H$. 
The routine proof of the following auxiliary claim is deferred to \refApp{apx:proof}.
\begin{claim}\label{cl:thm:CE} 
The following holds:\vspace{-0.5em}%
\begin{enumerate}
\renewcommand{\labelenumi}{\textup{(\roman{enumi})}}
\renewcommand{\theenumi}{\textup{(\roman{enumi})}}
\itemsep 0.125em \parskip 0em  \partopsep=0pt \parsep 0em 
	\item\label{union} For distinct~$G_1, G_2 \in \cL_H$ we have~$G_1 \cup G_2 \in \cL_H$. 
	\item\label{connected} If~$G \in \cL_H$ and~$J \in \cL_H$ covers~$G$, then the graph~$J \setminus G := J[V(J)\setminus V(G)]$ is connected. 
	\item\label{disjoint} If~$G \in \cL_H$ and distinct~$J_1, \dots, J_r \in \cL_H$ cover~$G$, then the~$J_i \setminus G$ are pairwise vertex-disjoint.\vspace{-0.125em}%
\end{enumerate}%
\end{claim} 
\begin{proof}[Proof-Sketch of \refL{lem:CE}]
Deferring the choices of the 
constants $C_H \ge 1 \ge c_H > 0$, let  
\begin{align}
\label{eq:CE:defz}
z &:= \Bigceil{ \bigpar{C_H\eps \mu_K}^{1/r}} ,\\
\label{eq:CE:defdelta}
\delta &:= c_H \min\{\eps,1\} .
\end{align}
Similarly as in the proof of \refL{lem:CE:fixed}, 
we expose the edges of~$\Gnp$ in three rounds: for $i \in [3]$ we insert each of the $\binom{n}{2}$ possible edges into~$\cE_i$ independently with probability~$p_i$, where 
\begin{equation}
\label{eq:CE:defpi}
p_1 := p_2 := \delta p \quad \text{ and } \quad p_3 := 1 - \frac{1-p}{(1 - p_1)(1-p_2)} = \left( 1 - O(\delta) \right) p .
\end{equation}

To establish the lower bound~\eqref{eq:thm:CE}, the strategy is to 
(i)~first use the $\cE_1$--edges to find one copy~$G'$ of~$G$. 
Next, we (ii)~partition the remaining vertex-set $[n] \setminus V(G')$ into $r$ sets~$V_1, \ldots, V_r$ of approximately equal sizes, 
and use the~$\cE_2$--edges to simultaneously extend~$G'$ to~$z$ copies of each~$J_i$ which (a)~embed~$V(J_i\setminus G)$ into~$V_i$, and (b)~are pairwise vertex-disjoint outside of~$V(G')$. 
This clearly enforces~$y:= z^r$ copies of~$K=J_1 \cup \cdots \cup J_r$ extending~$G'$ 
(by \refCl{cl:thm:CE}\ref{disjoint} all subgraphs~$J_i \setminus G$ are pairwise vertex-disjoint).  
Finally, we (iii)~use the $\cE_3$--edges to show that we can simultaneously
(a)~extend~$y=\Theta(C_H\eps \mu_K)$ of the aforementioned special copies of~$K$ via the $\cE_3$--edges 
to at least $2\eps \mu_H$ copies of~$H$, 
and (b)~also find at least $(1-\eps)\mu_H$ additional copies of~$H$ in~$\cE_3$ itself, 
so that we overall obtain~$X_H \ge 2\eps \mu_H + (1-\eps)\mu_H \ge (1+\eps)\mu_H$ copies of~$H$.

While some care is needed, the technical details of the outlined steps are mostly elementary, 
and thus deferred to \refApp{apx:proof}. Here we just mention that, 
analogously to 
\refL{lem:CE:fixed},  the probability of the `disjoint construction' from step~(ii) again gives the main contribution to our lower bound. 
In particular, by a more involved variant of the `enforcing $z$~neighbours' argument from~\eqref{eq:CE:prdisjoint},  
the aforementioned probability of step~(ii) that~$G'$ has~$z$ `non-overlapping extensions' to each~$J_i$
 will turn out to be 
(noting that~$\prod_{i \in [r]} n^{v_{J_i}-v_G} p^{e_{J_i}-e_G} \asymp \prod_{i \in [r]} (\mu_{J_i}/\mu_G) \asymp \mu_K/\mu_G $ by \refCl{cl:thm:CE}\ref{disjoint}, and that~$z^r \asymp \eps \mu_K$) 
roughly of form 
\begin{equation}\label{eq:thm:CE:disjoint1}
\prod_{i \in [r]} \binom{\binom{|V_i|}{v_{J_i}-v_G}}{z} {p_2}^{(e_{J_i}-e_G)z} 
\ge \biggpar{\prod_{i \in [r]}\frac{\Theta\bigpar{n^{v_{J_i}-v_G}(\delta p)^{e_{J_i}-e_G}}}{z}}^{z} 
\ge \biggpar{\frac{\Theta\bigpar{\prod_{i \in [r]}{\delta}^{e_{J_i}-e_G}} }{\eps \mu_G}}^{z}.
\end{equation}
Using~$\delta \asymp \eps$, $\eps^2 \gg 1/\Phi_H \ge 1/\mu_G$ and~$\mu_G \asymp (np^{m_H})^{v_G} \gg 1$ (by primality of~$G$), this in turn is at least 
\begin{equation}\label{eq:thm:CE:disjoint2}
\biggpar{\frac{\Theta\bigpar{\eps^{\sum_i (e_{J_i} - e_G) - 1}}}{\mu_G}}^{z} 
\ge \biggpar{\frac{1}{\mu_G^{\Theta(1)}}}^{z} \ge \bigpar{n p^{m_H}}^{-O(z)},
\end{equation}
making the right-hand side of inequality~\eqref{eq:thm:CE} plausible (see \refApp{apx:proof} for the full details). 
\end{proof}

\begin{proof}[Proof of \refT{thm:CE:cor}] 
Let~$\omega := np^{m_H}$ and $v_0 := \min_{F \in \cL_H} v_F$. 
Note that~$\Phi_H \asymp \omega^{v_0} \gg 1$ by \refR{rem:thm:CE}. 
Since the graph $K = J_1 \cup \dots \cup J_r$ is primal by \refCl{cl:thm:CE}\ref{union}, 
it follows easily that $\mu_K \asymp \omega^{v_K}$ (see, e.g.,~\eqref{eq:CE:density:F}).

We are now ready to prove~\eqref{eq:thm:CE:DK}.  
Since~$\Phi_H = O((\log n)^{v_0c_H})$ holds by assumption, we have~$\Phi_H \ll \log n \ll M_H \log (1/p)$ for~$c_H>0$ small enough. 
Since~$v_K/r < v_0$ holds by assumption, we also have~$\mu_K^{1/r}\log(np^{m_H}) \asymp \omega^{v_K/r}\log \omega \ll \omega^{v_0} \asymp \Phi_H$, 
so that inequality~\eqref{eq:thm:CE:DK} follows from \refL{lem:CE} (as~$\eps^2 \Phi_H \asymp \Phi_H \gg 1$).

We next turn to~\eqref{eq:thm:CE:SG}. 
Pick positive~$c \in \bigl(v_0/2-(rv_0 - v_K)/(2r-1), v_0/2\bigr)$, and define~$\lambda := \omega^c$.  
Using the variance estimate~\eqref{eq:sigmaH} we infer $\lambda \sigma_H/\mu_H \asymp \lambda /\Phi_H^{1/2} \asymp \omega^{c-v_0/2} \ll 1$ and thus~$\lambda \sigma_H \ll \mu_H$. 
Defining~$\eps := t/\mu_H = O(1)$, using~\eqref{eq:sigmaH} we also infer~$\eps^2\Phi_H \asymp t^2/\sigma_H^2 \ge \lambda^2 \gg 1$, so \refL{lem:CE} applies. 
Combining~$t^2/\sigma_H^2 \asymp \eps^2\Phi_H$  and~$\eps \ge \lambda \sigma_H/\mu_H \asymp \omega^{c-v_0/2}$ 
with~$\Phi_H \asymp \omega^{v_0}$ and~$\mu_K \asymp \omega^{v_K}$, it follows by choice of~$c$ that, say, 
\[
\bigpar{t^2/\sigma_H^2}^r \asymp \eps \cdot \eps^{2r-1} \bigpar{\Phi_H}^r \ge \eps \cdot \Omega\Bigpar{\omega^{(c-v_0/2)(2r-1)+r v_0}} \gg \eps \cdot \omega^{v_K} (\log \omega)^r \asymp \eps\mu_K (\log \omega)^r .
\]
This readily implies~$(\eps\mu_K)^{1/r} \log \omega \ll t^2/\sigma_H^2$, 
which in view of~\eqref{eq:thm:CE} establishes inequality~\eqref{eq:thm:CE:SG}. 
\end{proof}

\subsection{Optimizing the lower bound for the upper tail}\label{sec:covering:lowerbound}
In this subsection we optimize the lower bound~\eqref{eq:thm:CE} for the upper tail over all possible choices of~$G$ and~$K = J_1 \cup \dots \cup J_r$, 
restricting to the important case where~$\eps>0$ is fixed (as in \refP{pr:UT}); see \refL{lem:opt} below. 
To state our result, given~$G \in \cL_H$, let $J_1, \dots, J_{s(G)}$ be \emph{all} primals of~$H$ which cover~$G$, ordered by the increasing number of vertices (how the ties are broken is irrelevant for our purposes). 
Then, for graphs~$H$ which are not strictly balanced 
(which implies that there is~$G \in \cL_H$ with~$s(G) \ge 1$), we define  
\begin{equation}\label{eq:zetas}
	\zeta_H(G) := \min_{r \in [s(G)]}\biggset{\frac{v_G + \sum_{i=1}^r (v_{J_i} - v_{G})}{r}}
	\quad \text{ and } \quad 
	\zeta_H := \min_{G \in \cL_H} \bigset{\zeta_H(G) : \: s(G) \ge 1 }.
\end{equation}
\begin{lemma}\label{lem:opt}
For every graph~$H$ that is not strictly balanced 
there is a constant~$\beta_H > 0$ such that the following holds. 
For fixed~$\eps>0$ and any~$p=p(n) \in [0,1]$ with~$1 \ll n p^{m_H}  \le n^{\beta_H}$ we have 
 \begin{equation}\label{eq:lower_general}
\prob{X_H \ge (1+\eps)\E X_H} \ge \exp\biggcpar{- O\Bigpar{(np^{m_H})^{\zeta_H} \log (np^{m_H})}} ,
\end{equation}
where the implicit constant in~\eqref{eq:lower_general} may depend on~$\eps$ and~$H$. 
\end{lemma}
\begin{proof}
Fix arbitrary~$G \in \cL_H$ with $s(G) \ge 1$. 
Combining \refL{lem:CE} with $\mu_K^{1/r} \asymp (np^{m_H})^{v_K/r} \gg 1$ (cf.~the proof of \refT{thm:CE:cor}), 
it suffices to show that 
the minimum of~$v_{K_S}/|S|$ over all~$S \subseteq [s(G)]$ with~$S \neq \emptyset$ equals~$\zeta_H(G)$, 
where~$K_S := \cup_{i \in S} J_i$. 
By \refCl{cl:thm:CE}\ref{disjoint} the graphs~$J_i$ share no vertices except for those in $V(G)$,~so 
\begin{equation}\label{eq:vK}
v_{K_S} = v_G + \sum_{i \in S} (v_{J_i} - v_G).
\end{equation}
Recalling~$v_{J_1} \le \dots \le v_{J_{s(G)}}$, a moment's thought reveals  
that the minimum is always attained by one of the sets~$S \in \{ [1], [2], \dots, [s(G)]\}$, 
which establishes~$\min_S v_{K_S}/|S| = \zeta_H(G)$ and thus completes the proof. 
\end{proof}
\noindent 
%
%
It seems difficult to give a simple combinatorial description of the~$G \in \cL_H$ which minimize~$\zeta_H(G)$ in~\eqref{eq:zetas}. 
For balanced graphs~$H$ it is natural to first focus on the so-called `grading decomposition' $\{G_0, \dots, G_s\} \subseteq \cL_H$ of Bollob\'as and Wierman~\cite{BW89}, 
which determines the limit distribution of~$X_H$ at the appearance threshold (i.e., when $p \sim c n^{-1/m_H}$ for some~$c \in (0,\infty)$). 
Turning to the inductive definition of their decomposition, let~$G_0$ be the union of minimal primal subgraphs of~$H$. 
Then, given~$G_i \neq H$, let~$G_{i+1}$ be the union of all primal subgraphs covering~$G_i$. 
For balanced graphs~$H$ the resulting grading~$G_0 \subset \cdots \subset G_s$ always terminates with~$G_s=H$ 
(and \refCl{cl:thm:CE}\ref{union} implies~$G_j \in \cL_H$).  
In~\cite{BW89} the distribution of~$X_H$ at the threshold is then determined inductively:  
first counting $G_0$-subgraphs, then $G_1$-subgraphs that contain the~$G_0$ subgraphs, etc, continuing until all~$H$-subgraphs are counted. 
Moving a tiny bit above the appearance threshold (as in \refL{lem:opt}), 
it thus sounds plausible that the exponential decay of $\prob{X_H \ge (1+\eps)\E X_H}$ 
could potentially be determined by one of the `transitions' from~$G_i$ to~$G_{i+1}$, 
which in turn suggests that the minimum in~$\zeta_H$ might perhaps be attained by some~$G_j$. 
The following example shows that this speculation is~false.
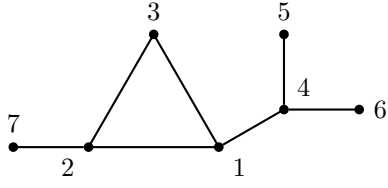
\begin{figure}[t]
\begin{center}
		\begin{tikzpicture}[thick,scale=1,decoration=brace]
    \tikzstyle{vertex}=[draw,circle,fill,fill opacity=1,minimum size=1pt, inner sep=1pt]
    \def\l{3}; 
    \def\r{2};
    \begin{scope}[rotate=-30]
	    \foreach \x in {1, 2, ..., \l}
	    {
		    \pgfmathsetmacro\angle{360/\l*(1 - \x)};
		    \draw (\angle:1) node[vertex,label={\angle -30 :$\x$}] (n\x) {} -- ( - 360/\l*\x:1) node () {};
	    }
	    \draw (n1) -- +(60:1) node[vertex,label={45:$4$}] (n4) {};
	    \draw (n4) -- +(120:1) node[vertex,label={90:$5$}] (n5) {};
	    \draw (n4) -- +(30:1) node[vertex,label={0:$6$}] (n6) {};
	    \draw (n2) -- +(210:1) node[vertex,label={90:$7$}] (n7) {};
    \end{scope}
	  \end{tikzpicture}\vspace{-1.5em}
  \end{center}
  \caption{The snail graph~$H$, which is balanced and satisfies~$m_H=1$. 
	\refE{ex:snail} demonstrates that no graph in the Bollob\'as--Wierman grading decomposition~$H[123] \subset H[12347] \subset H$ minimizes~$\zeta_H(G)$ in~\eqref{eq:zetas}.}
		\label{fig:snail}
\end{figure} 
\begin{example}\label{ex:snail}
Consider the graph $H$ in Figure \ref{fig:snail} with~$v_H=7$. 
Its primals (as vertex sets) are $123$, $1234$, $1237$, $12347$, $12345$, $12346$, $123456$, $123457$, $123467$, and~$1234567$.
Straightforward case checking reveals that~$\zeta_H$ is attained by~$1234$, 
which is covered by the three primals $12345$, $12346$, and~$12347$, 
so that~$\zeta_H(H[1234]) = \min \set{5/1,6/2,7/3}=7/3$. 
However, the Bollob\'as-Wierman grading decomposition is 
$G_0 := H[123] \subset G_1 := H[12347] \subset G_2 := H$, 
and both~$\zeta_H(G_0) = 5/2$ and~$\zeta_H(G_1) = 7/2$ are suboptimal.
\end{example}

\section{Concluding remarks}\label{sec:conclusion}
In this note we showed that the DeMarco--Kahn upper tail conjecture is false.  
Nevertheless we believe that its prediction is true when~$H$ is strictly balanced 
or~$p=p(n)$ is sufficiently above the appearance threshold. 
\begin{conjecture}
\refConj{conj:UT} is true for any strictly balanced graph~$H$. 
Furthermore, for any fixed~$\gamma>0$, \refConj{conj:UT} is true under the additional assumption~$p \ge n^{-1/m_H+\gamma}$. 
\end{conjecture}
\noindent
We leave it as an intriguing open problem to formulate an upper tail conjecture for graphs 
which are not strictly balanced (this would already be interesting for balanced graphs). 
Combining the new `locally-disjoint mechanism' based lower bound~\eqref{eq:lower_general} from \refL{lem:opt} 
with the previously known clustered/disjoint mechanism based lower bounds~\eqref{eq:lower:clustered}--\eqref{eq:lower:disjoint} from \refS{sec:discussion}, 
it is tempting to speculate that we might perhaps have 
\begin{equation}\label{eq:newconjecture}
	- \log \prob{X_H \ge (1 + \eps)\mu_H} = \Theta\Bigpar{\min\bigset{\Phi_H, \: M_H \log(1/p), \:  (np^{m_H})^{\zeta_H} \log(np^{m_H})}},
\end{equation}
which we believe to be correct for many graphs 
(e.g, for the graphs~$C_\ell^{+r}$ from \refS{sec:example}). 
However, the following result shows that the natural guess~\eqref{eq:newconjecture} is false 
for the balanced graphs~$H_r$ illustrated in \refF{fig:clusterexample}, 
indicating that for subgraph counts 
a general upper tail conjecture is most likely quite complicated. 
\begin{figure}[t]
\begin{center}
\begin{tikzpicture}[thick,scale=1.2,decoration=brace]
			\pgfdeclarelayer{bg}    
			\pgfsetlayers{bg,main} 
    \tikzstyle{vertex}=[draw,circle,fill,fill opacity=1,minimum size=1pt, inner sep=1pt]
	\foreach \x in {1, 2, ..., 6}
	{
		\draw (210 - 60*\x:1) node[vertex,label={210 - \x*60:$\x$}] (n\x) {} -- (150 - 60*\x:1) node () {};
	}
	\draw (n1) -- (n4);
	\draw (n2) -- (n5);
	\foreach \y in {7, 9, 10}
	{
	  \draw (-\y + 5, 0) node [vertex] (n\y) {} -- (-\y + 5 - 0.3, -0.5) node [vertex] () {} -- (-\y + 5 + 0.3, -0.5) node[vertex] () {} -- (n\y) -- (n1);
	}
	\draw node at (-3, -0.5) {$\dots$};
	\draw [decorate] (-1.5,-0.7) to node[auto] {$r - 1$ times} (-5.5,-0.7);
	\draw (n3) -- ++(1,0) node[vertex] (n10) {} 
	           -- ++(0.5,-0.5) node[vertex] (n11) {} 
	           -- ++(-0.5,-0.5) node[vertex] (n12) {} 
		   -- (n4);
	
\end{tikzpicture}
\vspace{-1.5em}
\end{center}
  \caption{The graph~$H_r$, which is balanced and satisfies $m_{H_r}=4/3$. 
	\refT{thm:badnews} illustrates that the upper tail behaviour of~$H_r$ is extremely complicated for~$r \ge 7$  (see also \refApp{apx:badnews}).}
		\label{fig:clusterexample}
\end{figure}
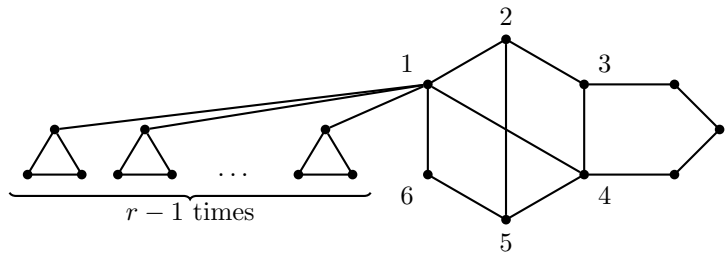
\begin{theorem}\label{thm:badnews} 
Let $\cH := \set{H_r: \: r \ge 7}$. 
For any~$H \in \cH$ there are constants~$1 > d_H > c_H > 0$ such that the following holds.  
For fixed~$\eps>0$ and any~$p=p(n) \in [0,1]$ with~$(\log n)^{c_H} \ll n p^{m_H}  \ll (\log n)^{d_H}$ we have 
\begin{equation}\label{eq:thm:badnews:LB}
-\log \prob{X_H \ge (1+\eps) \E X_H} = o\Bigpar{\min\bigset{\Phi_H, \: M_H \log(1/p), \:  (np^{m_H})^{\zeta_H} \log(np^{m_H})}} .
\end{equation}
\end{theorem}
\noindent
The proof of inequality~\eqref{eq:thm:badnews:LB} is based on the 
observation that different kinds of extensions (for~$H_r$ from \refF{fig:clusterexample} the dangling triangle and the rooted path) 
can have different ranges of~$p=p(n)$ where the disjoint mechanism beats the clustered one, 
which means that in some transitional range of~$p=p(n)$ a mixture of both mechanisms can potentially give better bounds (which turns out to be the case for~$H_r$). 
More precisely, adapting the framework of \refL{lem:CE} for~$H=H_r$ with~$G:=H[123456]$ and~$K :=H$, 
after planting one copy of~$G$ here the idea is to (a)~enforce~$z$ vertex-disjoint triangles which are each connected to vertex~$1$ of~$G$, 
and (b)~enforce at least~$z^*$ clustered copies of $5$-vertex paths with endvertices~$3,4$ of~$G$ 
(by planting a complete bipartite graph which connects a fixed vertex-set~$U$ of size~$2\sqrt{z^*}$ 
with the vertex-set~$\set{w,3,4}$, where the extra vertex~$w \not\in V(G) \cup U$ is also fixed). 
Analyzing these two mechanisms, it turns out that this way we obtain at least~$\binom{z}{r-1} \cdot z^*$ copies of~$H_r$ with probability at least~$(n p^{m_H})^{-O(z)} \cdot p^{\Theta(\sqrt{z^*})}$, 
which for suitable~$z \ll \mu_{H}^{1/r} \ll z^*$ and~$p=p(n)$ eventually gives inequality~\eqref{eq:thm:badnews:LB}; see \refApp{apx:badnews} for the details. 

Of course, one could augment~\eqref{eq:newconjecture} by the above-discussed new mix of the disjoint/clustered mechanisms 
(by adapting Lemmas~\ref{lem:CE} and~\ref{lem:opt}), 
but we are not sure if the resulting bound would be optimal (in general).

Finally, it would also be interesting to explore if Stein's method, 
large deviation theory (possibly after altering the variational problem from~\cite{CV11,CD16,Eldan16}), 
or some other probabilistic approach could yield an educated guess 
for the solution to the upper tail problem (\refP{pr:UT}) 
close to the appearance threshold~$n^{-1/m_H}$.

\bigskip{\noindent\bf Acknowledgements.} 
We thank the referees for helpful suggestions.

\begin{appendix}

\section{Appendix: Proof of \refT{thm:badnews}}\label{apx:badnews}
\begin{proof}[Proof of \refT{thm:badnews}] 
Fix~$H=H_r \in \cH$, with~$v_H = 3r+6$. Let~$\omega := n p^{m_H}$ and~$\gamma := 1/r^3$. 
Define 
\[
c_H := \frac{2}{v_H/r-(r-1)\gamma} 
\qquad \text{ and } \qquad 
d_H := \frac{1}{v_H/r-2+\gamma/2} , 
\]
noting that~$c_H < d_H$ as~$\gamma < (4-v_H/r)/r = (1-6/r)/r$. 
Since~$G$ has the smallest number of vertices among primals, we obtain~$\Phi_H \asymp \omega^6$ by \refR{rem:thm:CE}. 
Using~$1 \ll \omega \le n^{o(1)}$ and~$m_H \le \Delta_H/2$, it is not difficult to verify that~$M_H = \min_{G \subseteq H: e_G \ge 1} \mu_G^{v_G/\alpha_G^*}$ 
and thus~$M_H \asymp \omega^{\min_{F \in \cL_H} v_F/\al{F}}$~holds 
(e.g., by combining~\eqref{eq:CE:density:F}--\eqref{eq:CE:density:F:nonprimal} with~$1/\al{F} \in [1/v_F,1]$). 
Since every~$F \in \cL_{H}$ is a union of~$G$ and some (possibly empty) subset of the~$J_i$, using~\cite[Proposition A.4]{JOR} it turns out that~$\al{F} = v_F/2$, 
so~$M_H \asymp \omega^2$. 
It is routine to check that~$\zeta_H = v_H/r = 3+6/r$. 
It follows that~$\Phi_H \gg \omega^{\zeta_H} \log \omega$ and~$M_H \log(1/p)/\omega^{\zeta_H} \asymp (\log n)/\omega^{v_H/r-2} \gg \log \omega$, 
so the minimum in~\eqref{eq:thm:badnews:LB} satisfies 
\begin{equation}\label{eq:badnewsexponent}
\min\bigset{\Phi_H, \: M_H \log(1/p), \:  (np^{m_H})^{\zeta_H} \log(np^{m_H})} = \omega^{v_H/r}\log \omega .
\end{equation}

We are now ready to establish~\eqref{eq:thm:badnews:LB} by adapting the proof of \refL{lem:CE:fixed}, 
exposing the edges of~$G_{n,p}$ via~$\cE_1 \cup \cE_2$ in two independent rounds with edge-probabilities~$p_2 := p/2$ and $p_1 \in [p/2,p]$. 
For the desired lower bound, the strategy is to (i)~first use the $\cE_1$--edges to find one copy $G'$ of $G:=H[123456]$, where the vertices~$v_j$ of~$G'$ correspond to vertices~$j$ of $G$ (see  \refF{fig:clusterexample}).  
Next we (ii)~partition the vertex-set~$[n]\setminus V(G') = V_1 \cup V_2$ into two sets with~$|V_i| \approx n/2$, 
and then use the $\cE_2$--edges to simultaneously (a)~create~$z$ vertex-disjoint triangles in~$V_1$, which are each connected to vertex~$v_1$ of~$G'$, 
and (b)~create~$z^*$ `clustered' copies of a~$5$-vertex-path whose internal vertices are in~$V_2$ and whose endpoints are~$v_3,v_4$ of~$G'$.   
This together enforces at least~$\binom{z}{r-1} \cdot z^* > (1+\eps)\mu_H$ copies of~$H=H_r$ extending~$G'$ (see \refF{fig:clusterexample}), where 
\begin{align*}
z := \Bigceil{r((1+\eps) \mu_H)^{1/r} \omega^{- \gamma}} \asymp \omega^{v_H/r-\gamma} 
\quad \text{ and } \quad 
z^* := \Bigceil{((1+\eps)\mu_H)^{1/r} \omega^{(r-1)\gamma}} \asymp \omega^{v_H/r+(r-1)\gamma} . 
\end{align*}
Turning to the details, in step~(i) we find with probability~$1-o(1)$ at least one copy of~$G:=H[123456]$ in~$\cE_1$, 
since~$m_G=4/3=m_H$ and~$p_1 \ge p/2 \gg n^{-1/m_H}$ is above the appearance threshold. 
For step~(ii), we henceforth condition on the edge-set~$\cE_1$ and fix one copy~$G'$ of~$G$ in~$\cE_1$. 
Mimicking the calculations leading to~\eqref{eq:CE:prJ}--\eqref{eq:CE:prJ:harris} in \refApp{apx:proof}, 
it turns out that the probability of step~(ii).(a) is at least
\begin{equation}\label{eq:badnews:step2a}
\frac{1}{z!}
\prod_{0 \le j < z} \biggsqpar{\binom{|V_1| - 3j}{3} p_2^{4}} \cdot \omega^{-o(z)} 
\ge \biggpar{\frac{\Theta(n^{3}p^4 \omega^{-o(1)})}{z}}^{z} 
\ge \biggpar{\frac{\Theta(\omega^{3-o(1)})}{\omega^{v_H/r-\gamma}}}^z
\ge \omega^{-o(\omega^{v_H/r})}, 
\end{equation}
where we used~$v_H/r=3+6/r$. 
Turning to step~(ii).(b), after fixing a vertex-set~$U \subseteq V_2$ of size $|U|=\ceil{2 \sqrt{z^*}}$ and a vertex~$w \in V_2 \setminus U$,   
we define~$F$ as the complete bipartite graph between~$U$ and~$\set{v_3,v_4,w}$. 
Note that the union of~$G'$ and~$F$ contains at least~$\binom{|U|}{2} \ge z^*$ different $5$-vertex-paths with endpoints~$v_3,v_4$ and internal vertices from~$V_2$. 
Recalling~$p_2 = n^{-1/m_H + o(1)}$, the probability of step~(ii).(b) is thus at least
\begin{equation}\label{eq:badnews:step2b}
\prob{F \subseteq \cE_2} = 
p_2^{3|U|} 
\ge n^{-\Theta(\sqrt{z^*})} \ge \omega^{-o(\omega^{v_H/r})}, 
\end{equation}
where we used~$\omega^{v_{H}/r}/\sqrt{z^*}\asymp \omega^{(v_H/r-(r-1)\gamma)/2} \ge \log n$. 
Noting that the step~(ii) events lower bounded by~\eqref{eq:badnews:step2a}--\eqref{eq:badnews:step2b} 
are independent (as they depend on disjoint edge-sets), 
it follows that $\prob{X_H \ge (1+\eps) \E X_H} \ge \omega^{-o(\omega^{v_H/r})}$, 
which together with~\eqref{eq:badnewsexponent} implies inequality~\eqref{eq:thm:badnews:LB}.  
\end{proof}

\section{Appendix: Proof of \refL{lem:CE} and \refCl{cl:thm:CE}}\label{apx:proof}
\begin{proof}[Proof of \refCl{cl:thm:CE}]
For property~\ref{union}, 
using~$e_{G_i}/v_{G_i} = m_H \ge e_{G_1 \cap G_2}/v_{G_1 \cap G_2}$ 
it routinely follows that
\begin{equation}\label{eq:G1G2union}
\frac{e_{G_1 \cup G_2}}{v_{G_1 \cup G_2}} = \frac{e_{G_1} + e_{G_2} - e_{G_1 \cap G_2}}{v_{G_1} + v_{G_2} - v_{G_1 \cap G_2}} \ge m_H ,
\end{equation}
which implies that~$G_1 \cup G_2 \subseteq H$ is primal.

For property~\ref{disjoint}, suppose that~$J_i\setminus G$ and $J_j \setminus G$ with~$i \neq j$ are not vertex-disjoint. 
Clearly~$G \subsetneq J_i \cap J_j \subsetneq J_i$. Since~$J_k$ covers~$G$, this implies~$e_{J_i \cap J_j}/v_{J_i \cap J_j} < m_H$. 
Since~$e_{J_k}/v_{J_k}=m_H$, analogously to~\eqref{eq:G1G2union} we infer~$e_{J_i \cup J_j}/v_{J_i \cup J_j} > m_H$, 
reaching the desired contradiction (since~$J_i \cup J_j \subseteq H$).

For property~\ref{connected}, suppose that~$J \setminus G$ is not connected. 
Then we can partition~$V(J \setminus G)=V(J) \setminus V(G)$ into two non-empty vertex-sets~$V_j$  
such that there are no edges between~$V_1$ and~$V_2$ in~$J$. 
Since the graphs~$F_j := J[V(G) \cup V_j]$ are not primal (as~$G \subsetneq F_j \subsetneq J$), we have $e_{F_j}/v_{F_j} < m_H = e_{G}/v_{G}$. 
It follows that
\[
\frac{e_{J}}{v_{J}} = \frac{e_{F_1} + e_{F_2} - e_G}{v_{F_1} + v_{F_2} - v_G}  < m_H ,
\]
reaching the desired contradiction (since~$J \subseteq H$ is primal).  
\end{proof}

\begin{proof}[Proof of \refL{lem:CE}]
We keep the setup from the sketch in \refS{sec:proofsketch}: in particular, 
we shall expose the edges of~$\Gnp$ via~$\cE_1 \cup \cE_2 \cup \cE_3$ in three independent rounds 
with edge-probabilities~$p_1=p_2=\delta p$ and~$p_3 = (1-O(\delta))p$, 
where~$\delta = c_H \min\{\eps,1\}$ and~$c_H \le 1$. 
Adding an extra initial reduction step, we claim that  
it suffices to prove \refL{lem:CE} 
for graphs~$K=J_1 \cup \cdots \cup J_r$ which satisfy, for all $i \in [r]$, 
\begin{equation}\label{eq:reduction}
\mu_{J_i}/\mu_G \le (\eps\mu_K)^{1/r} . 
\end{equation}
To see that this implies \refL{lem:CE} for arbitrary~$K=J_1 \cup \cdots \cup J_r$, 
we use induction on the number of~$J_1, \ldots, J_r$ 
(formally allowing the implicit constant in inequality~\eqref{eq:thm:CE} to depend on~$1 \le r \le v_H$). 
The base case~$r=1$ is immediate, since~\eqref{eq:reduction} always holds 
due to~$(\eps \mu_K)^{1/r} = \eps \mu_G \cdot (\mu_{J_1}/\mu_G)$ and~$\eps \mu_G \ge \eps \Phi_H \gg \eps^{-1} = \Omega(1)$. 
For~$r \ge 2$ it suffices to consider the case where~\eqref{eq:reduction} fails for some~$i \in [r]$. 
Set~$K':= \bigcup_{j \neq i} J_j$. 
Applying induction (with~$K$ replaced by~$K'$, and thus~$r$ replaced by~$r-1$), 
the lower bound~\eqref{eq:thm:CE} holds with~$(\eps\mu_{K'})^{1/(r-1)}\log(np^{m_H})$ in the exponent. 
It thus remains to check that 
\begin{equation}\label{eq:betterfewer}
	(\eps \mu_{K'})^{1/(r-1)} = O\bigpar{(\eps\mu_{K})^{1/r}}.
\end{equation}
Using \refCl{cl:thm:CE}\ref{disjoint} we obtain $\mu_{K} \asymp \mu_{K'} \cdot \mu_{J_i}/\mu_G$. 
Since we assumed that~\eqref{eq:reduction} fails (i.e., that $\mu_{J_i}/\mu_G > (\eps\mu_K)^{1/r}$ holds) 
we infer~$\eps\mu_{K'} \asymp \eps\mu_{K} \cdot \mu_G / \mu_{J_i}=O((\eps \mu_K)^{1-1/r})$ 
and thus establish \eqref{eq:betterfewer}, completing the proof of the claimed reduction.

To facilitate our three-step proof strategy, we henceforth assume that~\eqref{eq:reduction} holds for all~$i \in [r]$.
Furthermore, we fix an ordering~$u_{1}, \ldots, u_{v_H}$ of the vertices of~$H$ 
such that the first~$v_G$ vertices are vertices of~$G$, 
the following~$v_{J_1} - v_G$ vertices are vertices of~$J_1 \setminus G$, followed by the vertices of~$J_2 \setminus G$, and so on up to $J_r \setminus G$ 
(this is possible since the subgraphs~$J_i \setminus G$ are pairwise vertex-disjoint, see \refCl{cl:thm:CE}\ref{disjoint}), 
while the final~$v_H - v_K$ vertices are the remaining vertices of~$H \setminus K$. 
We also introduce the shorthand notation
\begin{equation}\label{eq:CE:omega}
\omega := n p^{m_H} \quad \text{ with } \quad 1 \ll \omega \le n^{\beta_H} .
\end{equation}
We assume~$\beta_H < 1/v_H$, so that every primal subgraph~$F \subseteq H$ satisfies 
\begin{equation}\label{eq:CE:density:F}
\mu_F \asymp \bigpar{n p^{e_F/v_F}}^{v_F} = \bigpar{np^{m_H}}^{v_F} = \omega^{v_F} \le n^{v_F\beta_H} \ll n .  
\end{equation}
Furthermore, for any non-primal subgraph~$F \subseteq H$ we have~$B_{F,H} := m_H - e_F/v_F>0$, 
so that, say, 
\begin{equation}\label{eq:CE:density:F:nonprimal}
\mu_F \asymp \bigpar{np^{m_H} \cdot p^{- (m_H-e_F/v_F)}}^{v_F} \ge \bigpar{\omega \cdot n^{B_{F,H}(1-\beta_H)/m_H}}^{v_F} \gg n^{2v_H^2\beta_H } \ge \omega^{2v_H^2}
\end{equation}
for $\beta_H > 0$ small enough (the ad hoc~$2v_H^2$-term is convenient later on). 
From \eqref{eq:CE:omega}--\eqref{eq:CE:density:F:nonprimal} we easily deduce
\begin{equation}\label{eq:CE:density}
\Phi_G \ge \Phi_H \gg 1.
\end{equation}
%
Using $\eps^2\Phi_H \gg 1$ and \eqref{eq:CE:density:F} we obtain
\begin{equation}\label{eq:CE:prG:eps}
\delta \asymp\min\{\eps,1\} \gg (\Phi_H)^{-1/2} \ge (\mu_G)^{-1/2}= \Omega(\omega^{-v_G/2}) .
\end{equation}
Finally, recalling the definition~\eqref{eq:CE:defz} of~$z$, note that $\eps^2\Phi_H \gg 1$ and $\eps=O(1)$ imply $z^r \asymp \eps \mu_K \ge \eps \Phi_H \gg \eps^{-1} = \Omega(1)$ and~$z^r = O(\mu_K)$. 
Since~$K \subseteq H$ is primal (by \refCl{cl:thm:CE}\ref{union}), using~\eqref{eq:CE:density:F} it follows that 
\begin{equation}\label{eq:CE:z}
1 \ll z = O(\omega^{v_K/r}) \ll n^{1/r}.
\end{equation}

Turning to the technical details of step~(i), let~$\sX_G$ be the number of copies of~$G$ in~$\cE_1$. 
We claim that 
\begin{equation}\label{eq:CE:prG}
\Pr(\sX_G \ge 1) \gg \omega^{-v_G e_G} .
\end{equation}
For the proof we use a version of the Paley--Zygmund inequality (see, e.g.,~\cite[(3.3)--(3.4)]{JLR}) 
and the standard estimate $\Var \sX_G/(\E \sX_G)^2 \asymp 1/\Phi_G(n,p_1)$ (see, e.g.,~\cite[Lemma~3.5]{JLR}), 
so that~$p_1=\delta p$ and~$\delta \le 1$ imply  
\begin{equation*}
\Pr(\sX_G \ge 1) \ge \frac{(\E \sX_G)^2}{(\E \sX_G)^2 + \Var \sX_G} \asymp \min\bigcpar{1, \: \Phi_G(n,p_1)} \ge \min\bigcpar{1, \: \delta^{e_G}\Phi_G}.
\end{equation*}
Now inequality~\eqref{eq:CE:prG} follows, since~$\delta \gg \omega^{-v_G/2}$ by~\eqref{eq:CE:prG:eps} and~$\Phi_G \gg 1$ by~\eqref{eq:CE:density}.

For step~(ii), we henceforth condition on the edge-set~$\cE_1$, and assume that~$\sX_G \ge 1$. 
We also fix an \emph{ordered} copy~$G'$ of~$G$ in~$\cE_1$, i.e., a copy of~$G$ with~$E(G') \subseteq \cE_1$ and an ordering $u'_{1}, \ldots, u'_{v_G}$ of $V(G')$ that is consistent with the above-fixed ordering $u_{1}, \ldots, u_{v_G}$ of~$G$ (i.e., the injection~$u_j \mapsto u'_j$ maps edges of~$E(G)$ into edges of~$\cE_1$). 
We partition~$[n] \setminus V(G')$ into~$r$ vertex-sets~$V_1, \dots, V_r$ of approximately equal sizes~$n_i := |V_i| \approx n/r$. 
We say that an $(e_{J_i} - e_G)$-element edge-set~$\cS \subseteq \binom{V_i \cup V(G')}{2} \setminus \binom{V(G')}{2}$ is an \emph{$(G',J_i)$-edge-extension} if there is an injection from~$V(J_i)$ to 
$W(\cS):=V(G') \cup \bigcup_{f \in \cS} f$ 
 with~$u_j \mapsto u'_j$ for~$j \in [v_G]$ that maps every edge~$E(J_i) \setminus E(G)$ to an edge in~$\cS$ (this definition makes sense since~$J_i \setminus G = J_i[V(J_i) \setminus V(G)]$ contains no isolated vertices, see \refCl{cl:thm:CE}\ref{connected}). 
Note that~$|W(\cS) \setminus V(G')| = v_{J_i}-v_G$, and that $\cS \cup E(G')$ corresponds to (the edge-set of) a copy of~$J_i$ which contains~$G'$. 
Let~$Z_{G',J_i}$ be the number of  $(G',{J_i})$-edge-extensions~$\cS \subseteq \cE_2$. 
Noting that the random variables~$Z_{G',J_i}, i \in [r]$ depend on disjoint sets of independent $\cE_2$-edges, we infer
\begin{equation}\label{eq:CE:prZ}
\Pr(Z_{G',J_i} = z \text{ for all $i \in [r]$} \mid \cE_1 )  = \prod_{i \in [r]} \Pr(Z_{G',J_i} = z \mid \cE_1 ). 
\end{equation}
Fix~$i \in [r]$. We claim that 
\begin{equation}\label{eq:CE:prJ}
\Pr(Z_{G',J_i} = z \mid \cE_1 ) \ge \omega^{- O(z)}. 
\end{equation}
The following proof of~\eqref{eq:CE:prJ} is fairly standard (similar to, e.g.,~\cite[Proposition~9.1]{DKcliques}, \cite[Theorem~1]{Sileikis2012} or~\cite[Lemma~23]{APUT}), 
and we shall omit the conditioning on~$\cE_1$ from our notation to avoid clutter. 
Let~$\fS_i$ denote the set of all~$(G',J_i)$-edge-extensions~$\cS$.  
Since~$\cS \subseteq \binom{V_i \cup V(G')}{2} \setminus \binom{V(G')}{2}$ and~$z \ll n$ by~\eqref{eq:CE:z}, 
the number of $z$-element collections~$\cC \subseteq \fS_i$ of edge-extensions with pairwise disjoint vertex-sets~$W(\cS) \setminus V(G')$ is 
at least
\begin{equation}\label{eq:CE:prJ:bound}
\frac{1}{z!}
\prod_{0 \le j < z} \binom{n_i - j (v_{J_i}-v_G)}{v_{J_i}-v_G}
\ge \frac{1}{z!} \biggsqpar{\Bigpar{\frac{n_i - z (v_{J_i}-v_G)}{v_{J_i}-v_G}}^{v_{J_i}-v_G}}^z
\ge \biggpar{\frac{\Theta(n^{v_{J_i}-v_G})}{z}}^{z} .
\end{equation}
For any such collection~$\cC$, for brevity we introduce the events
\begin{equation*}
\cI_\cC := \set{\text{$\cE_2$ contains all $\cS \in \cC$}} \quad \text{ and } \quad 
\cD_\cC := \set{\text{$\cE_2$ contains no $\cS \in \fS_i \setminus \cC$}}. 
\end{equation*}
We trivially have $\Pr(\cI_\cC) \ge p_2^{(e_{J_i}-e_G)z}$ (in fact, this holds with equality), 
and defer the proof of 
\begin{equation}\label{eq:CE:prJ:harris}
\Pr(\cD_\cC \mid \cI_{\cC}) \ge \omega^{-o(z)}.
\end{equation}
Since there are at least~\eqref{eq:CE:prJ:bound} many such collections~$\cC$, 
using disjointness of the events~$\cI_\cC \cap \cD_\cC$ we obtain 
\begin{equation*}
\Pr(Z_{G',{J_i}} = z) \ge \sum_{\cC} \Pr(\cI_\cC) \Pr(\cD_\cC \mid \cI_{\cC}) \ge \biggpar{\frac{n^{v_{J_i}-v_G}p_2^{e_{J_i}-e_G} \omega^{-o(1)}}{z}}^z .
\end{equation*}
Note that~\eqref{eq:CE:density:F} gives~$\mu_{J_i}/\mu_G \asymp \omega^{v_{J_i}-v_G}$. 
Since~$\delta \gg \omega^{-v_G/2}$ by~\eqref{eq:CE:prG:eps} and~$z = O(\omega^{v_K/r})$ by~\eqref{eq:CE:z}, we infer 
\begin{equation*}
\frac{n^{v_{J_i}-v_G}p_2^{e_{J_i}-e_G}}{z} \asymp \frac{\mu_{J_i}}{\mu_G} \cdot \frac{\delta^{e_{J_i}-e_G}}{z}  \ge \omega^{-\Theta(1)} ,
\end{equation*}
and (recalling that we omitted the conditioning on~$\cE_1$ from our notation) inequality~\eqref{eq:CE:prJ} follows. 
It remains to give the deferred proof of estimate~\eqref{eq:CE:prJ:harris}. 
To this end observe that 
\[
\cD_{\cC}= \bigcap_{\cS \in \fS_i \setminus \cC}\set{\cS \not\subseteq\cE_2} \quad \text{ and } \quad \cI_{\cC}= \set{E_\cC \subseteq \cE_2} \qquad \text{ with } \qquad E_\cC := \bigcup_{\cS \in \cC} \cS. 
\]
Noting that the $\set{\cS \setminus E_{\cC} \not\subseteq \cE_2}$ are all decreasing events with respect to the independent~$\cE_2$-edge indicators, 
using Harris' inequality~\cite{Harris1960} (a special case of the FKG-inequality) it follows that 
\begin{equation}\label{eq:CE:prJ:harris1}
\Pr(\cD_\cC \mid \cI_{\cC}) = \Pr\Bigpar{\bigcap_{\cS \in \fS_i \setminus \cC}\set{\cS \setminus E_{\cC} \not\subseteq\cE_2}} \ge 
\prod_{\cS \in \fS_i \setminus \cC}\Pr(\cS \setminus E_{\cC} \not\subseteq\cE_2) 
= \prod_{\cS \in \fS_i \setminus \cC}\Bigpar{1-p_2^{|\cS \setminus E_{\cC}|}} . 
\end{equation}
Recall that each edge-extension~$\cS \in \fS_i$ is isomorphic to~$E({J_i}) \setminus E(G)$. 
Combining that~${J_i} \setminus G = {J_i}[V({J_i}) \setminus V(G)]$ is connected (see \refCl{cl:thm:CE}\ref{connected}) 
with the fact that all vertex-sets~$W(\cS) \setminus V(G')$ with~$\cS \in \cC$ are pairwise disjoint, 
it follows that~$E_\cC$ contains no further edge-extension~$\cS \in \fS_i \setminus \cC$. 
Therefore in every factor in~\eqref{eq:CE:prJ:harris1} we have~$|\cS \setminus E_{\cC}| \ge 1$ 
and thus~$\cS \setminus E_{\cC}$ is isomorphic to $E({J_i}) \setminus E(F)$ for some~$G \subseteq F \subsetneq {J_i}$. 
As~$p_2 \le p \ll 1$, $n_i \le n$ and~$|\cC|=z$, it follows that 
\begin{equation*}
-\log \Pr(\cD_\cC \mid \cI_{\cC}) 
\le 2 \sum_{\cS \in \fS_i \setminus \cC}p_2^{|\cS \setminus E_{\cC}|} 
\le 2 \sum_{G \subseteq F \subsetneq {J_i}} (v_{J_i}|\cC|)^{v_F-v_G} n^{v_{J_i}-v_F}p^{e_{J_i}-e_F}
= 
O\biggpar{\sum_{G \subseteq F \subsetneq {J_i}} z^{v_F-v_G}\frac{\mu_{J_i}}{\mu_F}}  .
\end{equation*}
Our initial reduction step ensures~$\mu_{J_i}/\mu_G \ll z \log \omega$, see~\eqref{eq:reduction} and~\eqref{eq:CE:defz}.  
Furthermore, \eqref{eq:CE:z} gives~$z = O(\omega^{v_K/r})$ and~\eqref{eq:CE:density:F} gives~$\mu_G \asymp \omega^{v_G}$. 
As no $G \subsetneq F \subsetneq {J_i}$ is primal (since~$J_i$ covers~$G$), using~\eqref{eq:CE:density:F:nonprimal} it follows that
\begin{equation*}
-\log \Pr(\cD_\cC \mid \cI_{\cC}) 
= O\biggpar{\frac{\mu_{J_i}}{\mu_G} \biggsqpar{ 1 + \sum_{G \subsetneq F \subsetneq {J_i}} \omega^{v_Fv_K/r} \frac{\omega^{v_G}}{\omega^{2v_H^2}}}}
\ll 
z \log \omega ,  
\end{equation*}
which completes the proof of~\eqref{eq:CE:prJ:harris} and thus inequality~\eqref{eq:CE:prJ}.

For the final step~(iii), we further (in addition to the conditioning on~$\cE_1$ from step~(ii) above) condition 
on the edge-set~$\cE_2$, assuming that~$Z_{G',J_i} = z$ for all~$i \in [r]$. Recalling that the subgraphs~$J_i \setminus G$ are vertex-disjoint (see \refCl{cl:thm:CE}\ref{disjoint}), note that if we pick any~$r$ copies of~$J_1, \ldots, J_r$ counted by~$Z_{G',J_1}, \ldots, Z_{G',J_r}$ (which are all vertex-disjoint outside of $G'$), 
then their union gives a copy of~$K = J_1 \cup \cdots \cup J_r$ (here it matters that the shared copy $G'$ is ordered). 
For each such copy of~$K$ we henceforth fix \emph{one} ordered copy~$K'$ with vertex-ordering $u'_{1}, \ldots, u'_{v_G}, u'_{v_{G}+1}, \ldots, u'_{v_K}$, say.
Let~$\cK$ denote the collection of all such ordered~$K'$ (each of which satisfies $E(K') \subseteq \cE_1 \cup \cE_2$), and define~$V(\cK)$ as the union of all their vertex-sets. 
Note that
\begin{equation}\label{eq:CE:cC}
|\cK| = z^r \asymp C_H \eps \mu_K \asymp C_H \eps n^{v_K}p^{e_K}.
\end{equation}
Given $K'\in \cK$, we say that a copy~$H'$ of~$H$ in $\cE_1 \cup \cE_2 \cup \cE_3$ is an \emph{$(K',H)$-extension} if~$H'$ 
contains the ordered copy~$K'$ with~$V(K')=\{u'_1, \ldots, u'_{v_K}\}$, 
satisfies $V(H') \setminus V(K') \subseteq [n] \setminus V(\cK)$, 
and there is an injection from~$V(H)$ to~$V(H')$ with~$u_j \mapsto u'_j$ for~$j \in [v_K]$ that maps every edge~$E(H) \setminus E(K)$ to an edge in~$\cE_3$. 
Let~$X'_H$ denote the number of copies of~$H$ which are $(K',H)$-extensions for some~$K' \in \cK$. 
Let~$X''_H$ denote the number of copies of~$H$ with vertices in~$[n] \setminus V(G')$ and all edges in~$\cE_3$. 
As the sets of $H$-copies counted by~$X'_H$ and~$X''_H$ are disjoint (the former contain~$G'$, and the latter share no vertices with~$G'$), 
we have~$X_H \ge X'_H + X''_H$. 
Noting that~$X'_H$ and~$X''_H$ are both increasing functions of the independent $\cE_3$-edge indicators,  
using Harris' inequality it follows that 
\begin{equation}\label{eq:CE:prH}
\begin{split}
\Pr(X_H \ge (1+\eps)\mu_H \mid \cE_1, \cE_2) 
& \ge \Pr(X'_H \ge 2\eps \mu_H\mid \cE_1, \cE_2) \cdot \Pr(X''_H \ge (1-\eps) \mu_H  \mid \cE_1, \cE_2) .
\end{split}
\end{equation}
To establish inequality~\eqref{eq:thm:CE} it thus suffices to prove 
\begin{align}
\label{eq:CE:prH1}
\Pr(X'_H \ge 2\eps \mu_H \mid \cE_1, \cE_2) & \gg \omega^{-v_K},\\
\label{eq:CE:prH2}
\Pr(X''_H \ge (1-\eps) \mu_H \mid \cE_1, \cE_2) & = 1-o(1) . 
\end{align}
Indeed, since we conditioned on $\cE_1$ satisfying~$\sX_G \ge 1$ and $\cE_2$ satisfying~$Z_{G',J_i} = z$ for all~$i \in [r]$, 
by combining~\eqref{eq:CE:prH}--\eqref{eq:CE:prH2} 
with estimates~\eqref{eq:CE:prG} and~\eqref{eq:CE:prZ}--\eqref{eq:CE:prJ}, 
then inequality~\eqref{eq:thm:CE} follows readily. 

In the remaining 
proofs of~\eqref{eq:CE:prH1}--\eqref{eq:CE:prH2} we shall again omit the conditioning (on $\cE_1,\cE_2$) from our notation. 
Turning to the crude estimate~\eqref{eq:CE:prH1}, we define~$Y_{K',H}$ as the number of $(K',H)$-extensions, so that 
\[
X'_H = \sum_{K' \in \cK}Y_{K',H} .
\]
Note that~\eqref{eq:CE:cC} and~\eqref{eq:CE:z} imply the rough bound~$|V(\cK)| \le v_K |\cK| \asymp z^r \ll n$, so that $|[n]\setminus V(\cK)| \asymp n$, say.  
Combining~\eqref{eq:CE:cC} with~$p_3 =(1-O(\delta))p \asymp p$ (which due to $\delta = c_H\min\set{\eps,1}$ holds for~$c_H>0$ sufficiently small)  and~$\mu_H = \Theta(n^{v_H}p^{e_H})$, it follows for $C_H >0$ sufficiently large that 
\[
\E X'_H =  \sum_{K' \in \cK} \E Y_{K',H} = |\cK| \cdot \Theta(n^{v_H-v_K}p_3^{e_H-e_K}) = C_H  \cdot \Theta(\eps\mu_H) \ge 4 \eps \mu_H .
\]
Similarly, for all $K'_1, K'_2 \in \cK$ we also have the routine upper bound 
\[
\E(Y_{K'_1,H} Y_{K'_2,H}) \le n^{v_H-v_K} p_3^{e_H-e_K} \sum_{K \subseteq F \subseteq H}  n^{v_H-v_F}p_3^{e_H-e_F} = \prod_{i \in [2]}\E Y_{K'_i,H}  \cdot O\biggpar{\sum_{K \subseteq F \subseteq H} \frac{\mu_K}{\mu_F}} .
\]
Since~$K$ is primal (see \refCl{cl:thm:CE}\ref{union}), by combining~$\mu_F \ge \Phi_H$ with estimates~\eqref{eq:CE:density:F} and~\eqref{eq:CE:density} it follows that 
\[
\E(X'_H)^2 = \sum_{K'_1, K'_2 \in \cK} \E(Y_{K'_1,H} Y_{K'_2,H}) \le (\E X'_H)^2 \cdot O(\mu_K/\Phi_H) \ll (\E X'_H)^2 \cdot \omega^{v_K} .
\]
Using a version of the Paley--Zygmund inequality (see, e.g.,~\cite[Lemma~3.2]{JOR}) we infer 
\[
\Pr(X'_H \ge 2\eps \mu_H) \ge \Pr(X'_H \ge \tfrac{1}{2}\E X'_H) \ge 
\frac{1}{4} \cdot \frac{(\E X'_H)^2}{\E(X'_H)^2}
\gg \omega^{-v_K},
\]
which (recalling that we omitted the conditioning on~$\cE_1,\cE_2$ from our notation) implies inequality~\eqref{eq:CE:prH1}.

Turning to the final estimate~\eqref{eq:CE:prH2}, 
for any $F \subseteq H$ with~$e_F \ge 1$ we define~$Y_F$ as the number of copies of~$F$ with vertex-set in~$[n]\setminus V(G')$ and edge-set in~$\cE_3$, 
so that~$X''_H=Y_H$. 
Note that~$Y_F$ has the same distribution as the number of copies of~$F$ in the (unconditional) binomial random graph $G_{n-v(G),p_3}$. 
Furthermore, $\delta \gg n^{-1}$ follows from~\eqref{eq:CE:prG:eps} and~\eqref{eq:CE:density:F}, with room to spare (since~$G$ is primal).  
Recalling the definitions of~$p_3= (1-O(\delta))p$ and~$\delta = c_H\min\set{\eps,1}$,  
for $c_H>0$ sufficiently small it thus is routine to see that 
\[
\frac{\E Y_F}{\mu_F} = \frac{\binom{n-v(G)}{v_F}}{\binom{n}{v_F}} \biggpar{\frac{p_3}{p}}^{e_F} 
= \bigpar{1-O(n^{-1})} \cdot \bigpar{1 - O(\delta)} 
\ge 1 -\eps/2 .
\]
Since also $\E Y_F \asymp \mu_F$, standard variance estimates for random graphs (see, e.g.,~\eqref{eq:sigmaH} or~\cite[Lemma~3.5]{JLR}) imply  
\[
\Var Y_H \asymp \frac{(\E Y_H)^2}{\min_{F \subseteq H : e_F \ge 1 } \E Y_F} \asymp \frac{(\mu_H)^2}{\Phi_H} .
\]
Using~$X''_H=Y_H$, Chebychev's inequality, and the assumption~$\eps^2\Phi_H \gg 1$ it follows that 
\[
\Pr(X''_H \le (1-\eps)\mu_H) \le \Pr(Y_H \le \E Y_H - \tfrac{1}{2}\eps\mu_H) \le \frac{\Var Y_H}{(\tfrac{1}{2}\eps\mu_H)^2} \asymp \frac{1}{\eps^2 \Phi_H} = o(1) ,
\]
which (as we omitted the conditioning on~$\cE_1,\cE_2$) completes the proof of~\eqref{eq:CE:prH1}--\eqref{eq:CE:prH2} and thus \refL{lem:CE}. 
\end{proof}


\end{appendix}


\small

\begin{thebibliography}{99}

\bibitem{RW15}
R.~Adamczak and P.~Wolff.
\newblock Concentration inequalities for non-{L}ipschitz functions with bounded derivatives of higher order.
\newblock {\em Probab.\ Theory Related Fields} \textbf{162} (2015), 531--586.

\bibitem{BGLZ17}
B.~Bhattacharya, S.~Ganguly, E.~Lubetzky, and Y.~Zhao.
\newblock Upper tails and independence polynomials in random graphs. 
\newblock {\em Adv.\ Math.} \textbf{319} (2017), 313--347. 

\bibitem{BB81}
B.~Bollob{\'a}s.
\newblock Threshold functions for small subgraphs.
\newblock {\em Math.\ Proc.\ Cambridge Philos.\ Soc.} {\bf 90} (1981), 197--206.

\bibitem{BW89}
B.~Bollob\'as and J.C.~Wierman.
\newblock Subgraph counts and containment probabilities of balanced and   unbalanced subgraphs in a large random graph.
\newblock In {\em Graph theory and its applications: {E}ast and {W}est ({J}inan, 1986)}, 
Annals of the New York Academy Vol.~{\bf 141}. pp.~763--70, New York Acad.\ Sci., New York (1989).
	
\bibitem{C12}
S.~Chatterjee.
\newblock The missing log in large deviations for triangle counts.
\newblock {\em Random Struct.\ Alg.} {\bf 40} (2012), 437--451.

\bibitem{CD16}
S.~Chatterjee and A.~Dembo.
\newblock Nonlinear large deviations.
\newblock {\em Adv.\ Math.} {\bf 299} (2016), 396--450. 

\bibitem{CV11}
S.~Chatterjee and S.R.S.~Varadhan.
\newblock The large deviation principle for the {E}rd{\H o}s-{R}\'enyi random graph.
\newblock {\em European J.\ Combin.} {\bf 32} (2011), 1000--1017.

\bibitem{DKtriangles}
B.~DeMarco and J.~Kahn.
\newblock Upper tails for triangles.
\newblock {\em Random Struct.\ Alg.} {\bf 40} (2012), 452--459.

\bibitem{DKcliques}
B.~DeMarco and J.~Kahn.
\newblock Tight upper tail bounds for cliques.
\newblock {\em Random Struct.\ Alg.} {\bf 41} (2012), 469--487.

\bibitem{Eldan16}
R.~{Eldan}.
\newblock {Gaussian-width gradient complexity, reverse log-Sobolev inequalities
  and nonlinear large deviations}.
\newblock {\em Geom.\ Funct.\ Anal.} {\bf 28} (2018) 1548--1596.

\bibitem{ER1960}
P.~Erd\H{o}s and A.~R\'enyi. 
\newblock On the evolution of random graphs. 
\newblock {\em Magyar Tud.\ Akad.\ Mat.\ Kutat\'o Int.\ K\"ozl} {\bf 5} (1960), 17--61.

\bibitem{KFRG}
A.~Frieze and M.~Karo{\'n}ski.
\newblock {\em Introduction to random graphs}.
\newblock Cambridge University Press, Cambridge (2016).

\bibitem{Harris1960}
T.E.~Harris.
\newblock A lower bound for the critical probability in a certain percolation process.
\newblock {\em Math.\ Proc.\ Cambridge Philos.\ Soc.} \textbf{56} (1960), 13--20.

\bibitem{J87}
S.~Janson.
\newblock Poisson convergence and {P}oisson processes with applications to random graphs.
\newblock {\em Stochastic Process.\ Appl.} \textbf{26} (1987), 1--30.

\bibitem{J90}
S.~Janson.
\newblock Poisson approximation for large deviations.
\newblock {\em Random Struct.\ Alg.} {\bf 1} (1990), 221--229.

\bibitem{JLR87}
S.~Janson, T.~{\L}uczak, and A.~Ruci{\'n}ski.
\newblock An exponential bound for the probability of nonexistence of a specified subgraph in a random graph.
\newblock In {\em Random graphs '87 ({P}ozna\'n, 1987)}, pp.~73--87, Wiley, Chichester (1990).

\bibitem{JLR}
S.~Janson, T.~{\L}uczak, and A.~Ruci{\'n}ski.
\newblock {\em Random graphs}.
\newblock {\em Wiley-Interscience Series in Discrete Mathematics and Optimization}. 
\newblock Wiley-Interscience, New York (2000).

\bibitem{JOR}
S.~Janson, K.~Oleszkiewicz, and A.~Ruci{\'n}ski.
\newblock Upper tails for subgraph counts in random graphs.
\newblock {\em Israel J.\ Math.} {\bf 142} (2004), 61--92.

\bibitem{DLP}
S.~Janson and A.~Ruci{\'n}ski.
\newblock The deletion method for upper tail estimates.
\newblock Preprint (2000).
\texttt{\detokenize{http://www2.math.uu.se/~svante/papers/sj135_ppt.pdf}}

\bibitem{UT}
S.~Janson and A.~Ruci{\'n}ski.
\newblock The infamous upper tail.
\newblock {\em Random Struct.\ Alg.} {\bf 20} (2002), 317--342.

\bibitem{DL}
S.~Janson and A.~Ruci{\'n}ski.
\newblock The deletion method for upper tail estimates.
\newblock {\em Combinatorica} {\bf 24} (2004), 615--640.

\bibitem{JWL}
S.~Janson and L.~Warnke.
\newblock {The lower tail: Poisson approximation revisited}.
\newblock {\em Random Struct.\ Alg.} {\bf 48} (2016), 219--246.

\bibitem{LZ17}
E.~Lubetzky and Y.~Zhao.
\newblock On the variational problem for upper tails in sparse random graphs.
\newblock {\em Random Struct.\ Alg.} {\bf 50} (2017), 420--436.

\bibitem{RW2012J}
O.~Riordan and L.~Warnke.
\newblock The {J}anson inequalities for general up-sets.
\newblock {\em Random Struct.\ Alg.} {\bf 46} (2015), 391--395.

\bibitem{RR94}
V.~R{\"o}dl and A.~Ruci{\'n}ski.
\newblock Random graphs with monochromatic triangles in every edge coloring.
\newblock {\em Random Struct.\ Alg.} {\bf 5} (1994), 253--270.

\bibitem{AR90}
A.~Ruci\'nski.
\newblock Small subgraphs of random graphs---a survey.
\newblock In {\em Random graphs '87 ({P}ozna\'n, 1987)}, pp.~283--303, Wiley, Chichester (1990).

\bibitem{Sileikis}
M.~{\v{S}}ileikis.
\newblock {\em Inequalities for Sums of Random Variables: a combinatorial perspective}.
\newblock PhD thesis, AMU Pozna\'n (2012). 
Available from \texttt{\detokenize{https://sites.google.com/site/matassileikis/}}

\bibitem{Sileikis2012}
M.~{\v{S}}ileikis.
\newblock On the upper tail of counts of strictly balanced subgraphs.
\newblock {\em Electron.\ J.\ Combin.} {\bf 19}~(1)~Paper~4 (2012), 14 pages. 

\bibitem{SW} 
M.~{\v{S}}ileikis and L.~Warnke. 
\newblock Counting extensions revisited. 
\newblock Manuscript (2019). 

\bibitem{Stars}
M.~\v{S}ileikis and L.~Warnke.
\newblock {Upper tail bounds for Stars}.
\newblock Preprint (2019). \texttt{arXiv:1901.10637}.

\bibitem{Vu2000}
V.H.~Vu.
\newblock On the concentration of multivariate polynomials with small expectation.
\newblock {\em Random Struct.\ Alg.} \textbf{16} (2000), 344--363.

\bibitem{Vu2001}
V.H.~Vu.
\newblock A large deviation result on the number of small subgraphs of a random graph.
\newblock {\em Combin.\ Probab.\ Comput.} {\bf 10} (2001), 79--94.

\bibitem{Vu2002}
V.H.~Vu.
\newblock Concentration of non-{L}ipschitz functions and applications.
\newblock {\em Random Struct.\ Alg.} {\bf 20} (2002), 262--316.

\bibitem{LW16}
L.~Warnke.
\newblock {On the missing log in upper tail estimates}.
\newblock {\em J.\ Combin.\ Theory Ser.~B}, to appear. \texttt{arXiv:1612.08561}.

\bibitem{APUT}
L.~Warnke.
\newblock {Upper tails for arithmetic progressions in random subsets}.
\newblock {\em Israel J.\ Math.} {\bf 221} (2017), 317--365.

\end{thebibliography}

\normalsize
\end{document}